\documentclass[reqno]{amsart}%11pt

\usepackage{amssymb,color}
\usepackage{amsmath}
\usepackage{amsfonts}
\usepackage{fouridx}
\usepackage[normalem]{ulem}
\usepackage{soul}
\usepackage[pdftex]{hyperref}
\usepackage{relsize}
\usepackage[utf8]{inputenc}
\usepackage{mathrsfs}
\usepackage[english]{babel}
\usepackage[totoc]{idxlayout}
\usepackage{fancyhdr}
\usepackage{paralist}
\usepackage{xcolor}
\usepackage[mathscr]{euscript}
\usepackage{mathtools}

\usepackage{todonotes}

\allowdisplaybreaks[3]
\newcommand\delc[1]{}

\newcommand\comcd[1]{}

\newcommand\del[1]{}
\newcommand\deln[1]{}
\newcommand\delr[1]{}

\newcommand\comad[1]{}

\newcommand\Greendel[1]{}
\newcommand\old[1]{}

\numberwithin{equation}{section}

\newtheorem{assumption}{Assumption}[section]

\def\old#1{}

\def\text#1{{\rm #1}}
\def\newold#1{}

\theoremstyle{plain}

\numberwithin{equation}{section}

\begin{document}

\title{Some dynamical properties of constrained Modified Swift-Hohenberg Equation}

\author{Saeed Ahmed}
\address{Department of Mathematics\\
Sukkur IBA University\\
Sindh Pakistan}
\email{saeed.msmaths21@iba-suk.edu.pk}

\author{Javed  Hussain}
\address{Department of Mathematics\\
Sukkur IBA University\\
Sindh Pakistan}
\email{javed.brohi@iba-suk.edu.pk}

\begin{abstract}
In this paper, we have studied the long-term behavior  for the projected deterministic constrained modified Swift-Hohenberg equation with constraints and Dirichlet boundary conditions. Specifically, using Lojasiewicz-Simon inequality, we have shown that  the global solution approaches an equilibrium state. Also, we have analyzed the rate at which the solution approaches equilibrium. Finally, we have proven the existence of a global attractor.
\end{abstract}
\date{\today}
\maketitle

%Theorem-like structures
%If you need new environments, define them here with the command \newtheorem{envname}{caption}. 

%Some environments such as definitions, theorems, lemmas or examples, have defined  in the list below.

\newtheorem{theorem}{Theorem}[section]
\newtheorem{lemma}[theorem]{Lemma}
\newtheorem{proposition}[theorem]{Proposition}
\newtheorem{corollary}[theorem]{Corollary}
\newtheorem{question}[theorem]{Question}

\theoremstyle{definition}
\newtheorem{definition}[theorem]{Definition}
\newtheorem{algorithm}[theorem]{Algorithm}
\newtheorem{conclusion}[theorem]{Conclusion}
\newtheorem{problem}[theorem]{Problem}

\theoremstyle{remark}
\newtheorem{remark}[theorem]{Remark}
\numberwithin{equation}{section}

\section{{\large Introduction}}

Doelman et al.\cite{Doelman} studied the following modified Swift-Hohenberg equation (MSHE) for the first time in 2003

\begin{align}{\label{dole}}
    u_{t}=-\alpha (1+ \Delta)^{2}u +\beta u-\gamma |\nabla u|^{2}-u^{3}
\end{align}
for a pattern formation with two unbounded spatial directions. Here, $\alpha >0, \beta, and ~\gamma $ are constants. When $\gamma=0$, equation (\ref{dole}) becomes the usual Swift-Hohenberg equation. The extra term $\gamma |\nabla u|^{2}$  reminiscent of the Kuramoto-Sivashinsky equation, which arises in
the study of various pattern formation phenomena involving some kind of phase
turbulence or phase transition \cite{Kuramoto, Siivashinsky} breaks the symmetry $u \rightarrow -u$.\\

From the formation of patterns to turbulence, the study of nonlinear differential equations has played a pivotal role in understanding physical phenomena. Of these equations, the Swift-Hohenberg equation has attracted scholars because it describes systems involving pattern-forming instabilities, such as Rayleigh–Benard convection and optical systems. Nonlinearity and constraints have been introduced to understand physical systems with more complex behaviors \cite{yochelis2006excitable, burke2006localized, kozyreff2006asymptotics}. More precisely, higher-order nonlinearities and domain constraints were introduced to comprehend localized patterns and complex spatiotemporal dynamics \cite{wetzel2010spatially, gomez2008swift, kirr2014pattern}. \\

Motivated by a broader range of physical systems, we are concerned with the following equation  obtained by modifying equation (\ref{dole}).
In the equation (\ref{SHEq}), $u(x,t)$ evolves under \textbf{negative bihormonic operator} $(-\Delta^{2}u)$, \textbf{ diffusion} $(2\Delta u)$, \textbf{a linear reaction term} $(au)$ and \textbf{a higher order non-linearity} $(u^{2n-1})$.

\begin{equation}{\label{SHEq}}
    u_{t}=-\Delta^{2}u+2\Delta u -au - u^{2n-1}
\end{equation}

In recent years, studying PDEs via attractor based approach  has attracted significant number of scholars. For instance, a global attractor  for MSHE was proven in \cite{Song,Xu}, the pull-back attractor in \cite{Park}, and the existence of a uniform attractor in \cite{Xu} for a non-autonomous MSHE were shown.
In addition, Mustafa Polat proved the Global attractor for a modified Swift–Hohenberg equation near the onset of instability with the Diriclet’s boundary conditions \cite{Polat},  Lingyu Song and others have shown global attractor of a modified Swift–Hohenberg equation in $H^{k}$ spaces using an iteration procedure, regularity estimates for the linear semigroups and a classical existence theorem of global attractor \cite {Song}. However, to the best of our knowledge, there is no work on the projected deterministic constrained modified Swift-Hohenberg equation in regard with global attractor and some dynamical properties such as the convergence of the global solution to an equilibrium. We aim to fill this gap.

This work is concerned with  the following projected deterministic constrained modified Swift-Hohenberg equation.

\begin{align}{\label{main_Prb_1}}
        \frac{\partial u}{\partial t} &=\pi _{u}(-\Delta^{2}u
        +2\Delta u -au -u^{2n-1}) \\
u(0) &=u_{0}(x),~~~~~~~~~~\text{for}~~x \in  \mathcal{O}  \notag 
    \end{align}

 where $\mathcal{O} \subset \mathcal{R}^{d}$ denotes a bounded domain with a smooth boundary $\partial \mathcal{O}$,  $n \in \mathbb{N} $  (or, in a general sense, an actual number such that $n>\frac{1}{2}$), and $u_{0} \in H_{0}^{1}(\mathcal{O}) \cap H^{2}(\mathcal{O})\cap \mathcal{M}$. \\

In this paper, we aim to study the long term behavior of the above problem. More precisely, we will demonstrate the convergence of the solution to equiilibrim and we will derive some decay estimates along with the existence of global attractor. \\
A plethora of work has been established on the convergence of parabolic flows to equilibrium with the different kinds of nonlinearities. To study these works, we suggest readers to \cite{Vishnevski,Matano,Simon,Lions,Hale, Haraux, Haraux2011,Haraux2009}. And for the detailed study of decay estimates,  we refer to \cite{Songu} and \cite{Haraux2003}. For the existence of global attractor for parabolic problems we refer to  \cite{Songu}, \cite{Temam}, \cite{Kapustyan} and \cite{Carvalho}.\\

This paper is organized as follows. Section 1 is about the introduction of the paper. Section 2 presents the necessary notations, function spaces,definitions, and preliminary results. Section 3 presents the proof that the orbits of the global solution to the proposed problem are precompact. In addition, the version the Lojasiewicz-Simon inequality is derived,  the convergence of the solution  to either an exponential or polynomial decay rate, depending on the values of the parameter $\theta$, is analyzed, and by using this inequality, the solution converges to equilibrium within the omega-limit set is shown. Section 4 studies the rate of convergence of the solution to equilibrium. Finally, in Section 5, we demonstrate the existence of a global attractor.

\section{{\Large Functional settings, Assumption and Preliminaries}}

\addtocontents{toc}{\protect\setcounter{tocdepth}{1}} 

We introduce some basic notions that we are going to use in this paper.

\begin{assumption}{\label{Ass_2.2.2}}

    Assume that the spaces , $(\mathcal{E}, \|.\|) $, $(\mathcal{V}, \|.\|_{\mathcal{V}}) $  and $(\mathcal{H}, |.|) $ are Banach spaces, where  

\begin{eqnarray*} 
\mathcal{H}&:=& L^2(\mathcal{O})\\
\mathcal{V} &:=& H_{0}^{1}(\mathcal{O}) \cap H^{2}(\mathcal{O})\\
\mathcal{E} &:=& D(A) =H_{0}^{2}(\mathcal{O}) \cap H^{4}(\mathcal{O}).
\end{eqnarray*} 

and  the following embeddings are dense and continuous: $$ \mathcal{E} \hookrightarrow \mathcal{V} \hookrightarrow  \mathcal{H}. $$

And the following  self-adjoint operator
 $A :  \mathcal{E} \rightarrow {L}^{2}(\mathcal{O})$ is defined as
 $$ A = \Delta^{2}-2\Delta ~~~~~ \text{and} ~~~\mathcal{O} \subset \mathbb{R}^{d} $$

Notice that, we use norm on $\mathcal{V}$ as:

\begin{align*}
    \|u\|^{2}_{\mathcal{V}} &= \|u\|^{2}_{H^{1}_{0}} +  \|u\|^{2}_{H^{2}} =\|u\|^{2}_{L^{2} (\mathcal{O})} +  2\|\nabla u\|^{2}_{L^{2} (\mathcal{O})} +   \|\Delta u\|^{2}_{L^{2} (\mathcal{O})} 
\end{align*}
 
\end{assumption}

\subsection{Hilbert manifold, and Orthogonal projection}

  \begin{definition}

    \cite{masiello1994variational} Assume that ${H}$ is a Hilbert space with inner product $\langle ~ \cdot, ~\cdot \rangle$, then the Hilbert manifold $\mathcal{M}$ is given as:
\begin{align*}
    \mathcal{M} =  \{~ u \in \mathcal{H}, &|u|_{\mathcal{H}}^{2}=1~\}     
\end{align*}   
\end{definition}

\begin{definition}

\cite{masiello1994variational}, \cite{hussain2015analysis} At a point $u~ \in ~{H}$, Tangent space $T_{u}\mathcal{M}$ is given as:
\begin{align*}
    T_{u}\mathcal{M}= \{~ h: ~~\langle h, u \rangle = 0~ \},~~~~~~~h \in {H}
\end{align*}  
\end{definition}

The orthogonal projection of $h$ onto $u$, $ \pi_{u}~: ~ \mathcal{H} \longrightarrow T_{u}\mathcal{M}$ is defined as
\begin{equation}{\label{lemma_Tangent}}
    \pi_{u}(h)= h-\langle h, u \rangle ~u, ~~~~~~h \in \mathcal{H}
\end{equation}  

\begin{remark}
  For any $T\geq 0$ we denote 
\begin{equation*}
   X_{T}:=L^{2}\left( 0,T;\mathcal{E}\right) \cap C\left( \left[ 0,T\right] ;\mathcal{V}\right) ,
\end{equation*} 
then it can be easily proven that  $\left( X_{T},\left\vert \cdot\right\vert _{X_{T}}\right) $  is the Banach space with the following norm, 
\begin{equation*}
\left\vert u\right\vert _{X_{T}}^{2}=\underset{p\in \lbrack 0,T]}{\sup }\left\Vert u(p)\right\Vert ^{2}
+\int_{0}^{T}\left\vert u(p)\right\vert_{\mathcal{E}}^{2}dp,\qquad u \in X_{T}
\end{equation*} 
\end{remark}

\begin{corollary}
    
If $ u \in \mathcal{E} \cap \mathcal{M}$ then   the projection of $-\Delta^{2}u+2\Delta u -au - u^{2n-1}$ under the map $\pi_{u}$ is:\\

$\pi _{u}(-\Delta^{2}u+2\Delta u -au - u^{2n-1})= -\Delta^{2}u+2 \Delta u  + \|  u\|^{2}_{{H}^{2}_{0}} ~u + 2\|  u\|^{2}_{{H}^{1}_{0}} ~u \notag  ~ +\| u\|^{2n}_{{L}^{2n}} u- u^{2n-1} $

\begin{proof}
    
Using (\ref{lemma_Tangent}) and integration by parts formula \cite{Haim}, we have:

 \begin{eqnarray*}
 &&\pi _{u}(-\Delta^{2}u+2\Delta u -au - u^{2n-1})\notag\\ &=&-\Delta^{2}u+2\Delta u -au - u^{2n-1}
+\langle \Delta^{2}u-2\Delta u +au + u^{2n-1}, u \rangle ~u \notag \\
&=&-\Delta^{2}u+2 \Delta u -au - u^{2n-1} + \langle \Delta^{2}u, u \rangle ~u -2\langle \Delta u, u \rangle ~u \notag +a\langle u, u  
\rangle ~u+\langle u^{2n-1}, u \rangle ~u \notag \\
&=&-\Delta^{2}u+2 \Delta u -au - u^{2n-1} + \langle \Delta u,  \Delta u \rangle ~u -2\langle - \nabla u, \nabla u \rangle ~u \notag ~+a\langle u, u  
\rangle ~u+\langle u^{2n-1}, u \rangle ~u \notag \\
&=&-\Delta^{2}u+2 \Delta u -au - u^{2n-1} + \| \Delta u\|^{2}_{{L}^{2}(\mathcal{O})} ~u + 2\| \nabla u\|^{2}_{{L}^{2}(\mathcal{O})} ~u \notag  +au+\| u\|^{2n}_{{L}^{2n}(\mathcal{O})} u \notag \\ 
&=&-\Delta^{2}u+2 \Delta u  + \|  u\|^{2}_{{H}^{2}_{0}} ~u + 2\|  u\|^{2}_{{H}^{1}_{0}} ~u \notag  ~ +\| u\|^{2n}_{{L}^{2n}} u- u^{2n-1}.
\end{eqnarray*}    
Therefore, the projection of $-\Delta^{2}u+2\Delta u +au + u^{2n-1}$ under map  $\pi_{u}$ is
\begin{eqnarray} {\label{Projection_U}}
      \pi _{u}(-\Delta^{2}u+2\Delta u-au - u^{2n-1}) \notag =  -  \Delta^{2}u+2 \Delta u  + \|  u\|^{2}_{{H}^{2}_{0}} ~u + 2\|    u\|^{2}_{{H}^{1}_{0}} ~u  +\| u\|^{2n}_{{L}^{2n}} u- u^{2n-1} 
\end{eqnarray} 

\end{proof}
\end{corollary}

The following theorem presents the global well-posedness of the problem (\ref{PB}), invariance of solution in the Hilbert manifold, energy equality, and some estimates. This theorem is a key result and will be used in this paper.

\begin{theorem}{\label{global_sol_thm}} \cite{JS}

Let  $\mathcal{E} \subset\mathcal{V} \subset \mathcal{H}$ satisfy the assumption (\ref{Ass_2.2.2}) and for any given $K>0$ there is $T'(K)$ such that $\|u_{0}\| \leq K$, where $u_{0} \in \mathcal{V}$, there is the unique local solution $ u : [0, T') \longrightarrow \mathcal{V}$ to the problem

 \begin{equation}{\label{PB}}
  \left\{
  \begin{aligned}
 \frac{\partial u}{\partial t} &=\pi _{u}(-\Delta^{2}u+2\Delta u -au - u^{2n-1}) \\
        &=-\Delta^{2} u+2\Delta u+\|  u\|^{2}_{{H}^{2}_{0}} ~u + 2\|    u\|^{2}_{{H}^{1}_{0}} ~u  +\| u\|^{2n}_{{L}^{2n}} u- u^{2n-1} =-Au + F(u(t)) \\ 
u(0) &=u_{0}. 
\end{aligned}
\right.
\end{equation}

In addition, the solution $u(t)$ stays on the manifold $\mathcal{M}$, that is, $u(t) \in \mathcal{M}$ for all $t\geq 0$. More importantly, we have an inequality.

\begin{align}\label{der}
   \mathcal{Y}(u(t))- \mathcal{Y}(u_{0}) &= -\int^{t}_{0} {\left\|u_{p}(p) \right\|_{L^{2}}^{2}} ~dp
\end{align} 

Which gives,

\begin{align*}
   \|u(t)\|_{\mathcal{V}}\leq  2~\mathcal{Y}(u_{0}), ~~~~~~\forall t \in [0,T) 
   \end{align*}

Where   $ \mathcal{Y}: \mathcal{V}\longrightarrow R $  is defined as: 
\begin{align*}
    \mathcal{Y}(u) = \frac{1}{2} \|u\|^{2}_{\mathcal{V}} - \frac{1}{2n} \|u\|^{2n}_{L^{2n}}, ~~~~ n \in N 
\end{align*} 
\end{theorem}

Now, we present some definitions and abstract results from Chapter 6 of \cite{Songu} regarding the convergence to equilibrium (stationary solutions) and global attractors.

\begin{corollary} \cite{Songu}  \label{2.2}
Suppose that $\mathcal{H}$ is a complete metric space and $S(t)$ is a nonlinear $C_0$ semigroup defined on $\mathcal{H}$. Let $x \in H$. If there is $t_0 \geq 0$ such that
$$
\{u(t):t> t_0\}:=\bigcup_{t>t_0} S(t) x
$$
is relatively compact in $\mathcal{H}$, then the $\Omega$-limit set $\Omega(x)$ is a compact, connected invariant set.   
\end{corollary}
\begin{theorem}\cite{Songu} Let $\Gamma: \mathbb{R}^N \rightarrow \mathbb{R}$ be an analytic function in a neighborhood of a point a in $\mathbb{R}^N$. Then there exists $\sigma>0$ and $0<0<\frac{1}{2}$ such that $\forall x \in \mathbb{R}^N$,
$$
\|x-a\|<\sigma  \implies \|\nabla \Gamma(x)\| \geqslant|\Gamma(x)-\Gamma(a)|^{1-\theta}.
$$   
\end{theorem}
\begin{definition}\cite{Songu} 
 Suppose that $\mathcal{H}$ is a complete metric space, and $S(t)$ is a nonlinear $C_0$-semigroup of operators defined on $\mathcal{H}$. set $\mathcal{A} \subset H$ is called an attractor if the following hold:\\
(i) $\mathcal{A}$ is an invariant set, i.e.,
$$
S(t) \mathcal{A}=\mathcal{A}, \quad \forall t \geq 0 .
$$
(ii) $\mathcal{A}$ possesses an open neighborhood $\mathcal{U}$ such that for any element $u_0 \in \mathcal{U}$, as $t \rightarrow+\infty, S(t) u_0$ converges to $\mathcal{A}$, i.e.,
$$
\operatorname{dist}\left(S(t) u_0, \mathcal{A}\right)=\inf _{y \in \mathcal{A}} d\left(S(t) u_0, y\right) \rightarrow 0 .
$$   
\end{definition}
\begin{definition}\cite{Songu} 
If $\mathcal{A}$ is a compact attractor, and it attracts bounded sets of $\mathcal{H}$, then $\mathcal{A}$ is called a global or universal attractor.  
\end{definition}
\begin{theorem}\cite{Songu} \label{attractor}
Suppose that $\mathcal{H}$ is a Banach space and $S(t)$ is a nonlinear $C_0$-semigroup defined on $\mathcal{H}$, satisfying the following conditions:\\
(i) there exists a bounded absorbing set $\mathcal{B}_0$;\\
(ii) for any bounded set $\mathcal{B}$, there is $t_0(\mathcal{B}) \geq 0$ depending on $\mathcal{B}$ such that
$$
\bigcup_{t \geq t_0(\mathcal{B})} S(t) \mathcal{B}
$$
is relatively compact in $\mathcal{H}$. Then $\mathcal{A}=\omega\left(\mathcal{B}_0\right)$ is a global attractor.    
\end{theorem}
Moreover, for the convenience of the reader, we present the following  \cite{temam2000navier}, \cite{hussain2015analysis} (Lemma 1.2, Chapter 3).

\begin{lemma}[Lemma III $1.2$, \cite{temam2000navier}, \cite{hussain2015analysis}]
\label{1-tem}Let $\mathcal{V},\mathcal{H}$ and $\mathcal{V}^{\prime }$ be three Hilbert spaces with $\mathcal{V}^{\prime }$ being the dual space of $\mathcal{V}$ and each included and dense in the
following one
\begin{equation*}
\mathcal{V}\hookrightarrow   \mathcal{H}\cong \mathcal{H}^{\prime }\hookrightarrow  \mathcal{V}^{\prime }.
\end{equation*}
If $u$ belongs to $\mathcal{L}^{2}(0,T;\mathcal{V})$ and its weak derivative $\frac{\partial u}{\partial t}$ belongs to $\mathcal{L}^{2}(0,T;\mathcal{V}^{\prime })$ then there exists $\widetilde{u}\in $ $\mathcal{L}^{2}(0,T;\mathcal{V})\cap C\left( \left[ 0,T\right] ;\mathcal{V}\right) $
such that $\widetilde{u}=u$ a.e. and we have the following energy equality:
\begin{equation*}
\left\vert u(t)\right\vert^{2}=\left\vert u_{0}\right\vert^{2}+2\int_{0}^{t}\left\langle u^{\prime }\left( s\right) ,u\left( s\right)\right\rangle ds, \, \mbox{ for all }  t\in \left[ 0,T\right].
\end{equation*}
\end{lemma}

\section{{\Large Lojasiewicz–Simon Inequality and Convergence to equilibrium}}

To prove that orbits of the global solution obtained in theorem \ref{global_sol_thm}  converges to an equilibrium (which is essentially is a solution stationary solution to problem \eqref{PB}), we should follow the following procedure. \\
Firstly,  show that orbit $\left\{ u(t): t\geq 1\right\}$ is precompact in $\mathcal{V}$ and the corresponding omega limit set $\Omega(u_0)$ is compact in $\mathcal{V}$. Secondly, prove a version of Lojasiewicz–Simon Inequality satisfied by the energy of the system. Finally,  prove the convergence of orbits of global solution converges, in $\mathcal{V}$-norm, to stationary solutions living inside $\Omega(u_0)$. \\

\begin{theorem}{\label{Pre_compact}}

Let the solution $u(t)$ follow theorem (\ref{global_sol_thm}); then, the orbit $\{ u(t); t\geq 1\}$ is pre-compact in $\mathcal{V}$.
\begin{proof}
    Recall that $\{ u(t); t\geq 1\}$  is precompact in $\mathcal{V}$ if it is bounded in $D(A^{\mu})$.\\ 
    
    What we are about to do is that  $\{ u(t); t\geq 1\}$ is bounded in $D(A^{\mu})$, ~~for $\mu > \frac{1}{2}$. \text{Where,}~$A= \Delta^{2}-2 \Delta$ \\
    \\
From the variation in the constant formula, we have:
\begin{align*}
    A^{\mu}u(t)= A^{\mu} e^{-At}u_{0}+ \int^{t}_{0} A^{\mu} e^{-A(t-p)} F(u(p)) ~dp
\end{align*}
Where $ F(u)=\|  u\|^{2}_{{H}^{2}_{0}} ~u + 2\|    u\|^{2}_{{H}^{1}_{0}} ~u  +\| u\|^{2n}_{{L}^{2n}} u- u^{2n-1}  $

\begin{align*}
\left|A^{\mu}u(t) \right|_{\mathcal{H}} & = \left|  A^{\mu} e^{-At}u_{0}+ \int^{t}_{0} A^{\mu} e^{-A(t-p)} F(u(p)) ~dp\right|_{\mathcal{H}}  \leq \left| A^{\mu} e^{-At}u_{0}\right|_{\mathcal{H}} + \left|\int^{t}_{0} A^{\mu} e^{-A(t-p)} F(u(p)) ~dp \right|_{\mathcal{H}}\\
 & \leq  \left|A^{\mu} e^{-At}u_{0} \right|_{\mathcal{H}} + \int^{t}_{0} \left|A^{\mu} e^{-A(t-p)} F(u(p))\right|_{\mathcal{H}} ~dp
\end{align*}

As by using proposition 1.23 in \cite{henry2006geometric} , we have $ \left|A^{\mu} e^{-At} \right|_{\mathcal{H}} \leq M_{\mu}t^{-\mu} e^{-\delta t}$\\
therefore, 
\begin{align*}
    \left|A^{\mu}u(t) \right|_{\mathcal{H}} & \leq M_{\mu}t^{-\mu} e^{-\delta t} \left|u_{0}\right|_{\mathcal{H}}+ \int^{t}_{0} M_{\mu} t^{t-\mu}e^{-\delta(t-p)} \left|F(u(p))\right|_{\mathcal{H}} ~dp
\end{align*}
As $u_{0} \in \mathcal{M}$ so $ \left|u_{0}\right|_{\mathcal{H}}=1 $,
\begin{align}{\label{Main_exp_compact}}
     \left|A^{\mu}u(t) \right|_{\mathcal{H}} & \leq M_{\mu}t^{-\mu} e^{-\delta t} + \int^{t}_{0} M_{\mu} (t-p)^{-\mu}e^{-\delta(t-p)} \left|F(u(p))\right|_{\mathcal{H}} ~dp
\end{align}
Now we will compute $\left|F(u(p))\right|_{\mathcal{H}} $\\

\begin{align*}
    \left|F(u)\right|_{\mathcal{H}}& =\left| \|  u\|^{2}_{{H}^{2}_{0}} ~u + 2\|    u\|^{2}_{{H}^{1}_{0}} ~u  +\| u\|^{2n}_{{L}^{2n}} u- u^{2n-1} \right|_{\mathcal{H}}  \leq  \|  u\|^{2}_{{H}^{2}_{0}} ~|u|_{\mathcal{H}} + 2\|    u\|^{2}_{{H}^{1}_{0}} ~|u|_{\mathcal{H}} +\| u\|^{2n}_{{L}^{2n}} |u|_{\mathcal{H}}+ \left| u^{2n-1} \right|_{\mathcal{H}}
\end{align*}
 Because $u \in \mathcal{M}$, $ \left|u\right|_{\mathcal{H}}=1 $. And $\mathcal{V} \hookleftarrow {{H}^{2}_{0}} $, $\mathcal{V} \hookleftarrow {{H}^{1}_{0}} $ and $\mathcal{V} \hookleftarrow {L^{2n}} $ is therefore $  \|  u\|^{2}_{{H}^{2}_{0}}  \leq \|  u\|^{2}_{\mathcal{V}}$, and $  \|  u\|^{2}_{{H}^{1}_{0}}  \leq \|  u\|^{2}_{\mathcal{V}}$  it follows:
\begin{align*}
    \left|F(u)\right|_{\mathcal{H}}& \leq \|  u\|^{2}_{\mathcal{V}}  + 2\|    u\|^{2}_{\mathcal{V}}  +\| u\|^{2n}_{L^{2n}} + \left| u^{2n-1} \right|_{\mathcal{H}}
\end{align*} 

From the energy function $  \mathcal{Y}(u) = \frac{1}{2} \|u\|^{2}_{\mathcal{V}} + \frac{1}{2n} \|u\|^{2n}_{L^{2n}}, ~~~~ n \in N  $ we obtain

\begin{align*}
    \|  u\|^{2}_{\mathcal{V}} \leq 2 ~\mathcal{Y}(u_{0}),  ~~\text{and}~~
     \|  u\|^{2n}_{L^{2n}} \leq 2n ~\mathcal{Y}(u_{0})
\end{align*}
Therefore, 
\begin{align}{\label{F(u(p)}}
    \left|F(u(p))\right|_{\mathcal{H}}& \leq (4+2n) ~\mathcal{Y}(u_{0})+ \left| u^{2n-1} \right|_{\mathcal{H}}
\end{align}

Now consider the expression $\left| u^{2n-1} \right|_{\mathcal{H}}$ \\

\begin{align*}
    \left| u^{2n-1} \right|^{2}_{\mathcal{H}} & = \int_{D}{\left(u^{2n-1}(p)\right)^{2}} ~dp \\
    \left| u^{2n-1} \right|_{\mathcal{H}} & =\left( \int_{D}{u^{4n-2}(p) ~dp }\right)^{\frac{1}{2}}= \left(\left( \int_{D}{u^{4n-2}(p) ~dp }\right)^{\frac{1}{4n-2}}\right)^{2n-1}\|u\|^{2n-1}_{L^{4n-2}}
\end{align*}
As $ \mathcal{V} \hookrightarrow L^{4n-2}$ so there is the constant  $K$ such that $ \|u\|^{2n-1}_{L^{4n-2}} \leq K^{2n-1} \|u\|^{2n-1}_{\mathcal{V}}$

\begin{align}{\label{embedig_4n-2}}
     \left| u^{2n-1} \right|_{\mathcal{H}} & \leq  K^{2n-1}\|u\|^{2n-1}_{\mathcal{V}}
\end{align}
By using (\ref{embedig_4n-2}), (\ref{F(u(p)}), and $\|u\|^{2n-1}_{\mathcal{V}} \leq 2^{2n-1} \left( \mathcal{Y}(u_{0})\right)^{2n-1}$, it follows that
 
\begin{align}{\label{F_final}}
    \left|F(u(p))\right|_{\mathcal{H}}& \leq (4+2n) ~\mathcal{Y}(u_{0})+ 2^{2n-1} K^{2n-1} \left( \mathcal{Y}(u_{0})\right)^{2n-1} : = N < \infty
\end{align}
Using (\ref{F_final}) and (\ref{Main_exp_compact}), it follows that 
\begin{align*}
    \left|A^{\mu}u(t) \right|_{\mathcal{H}} & \leq M_{\mu}t^{-\mu} e^{-\delta t} +M_{\mu}N \int^{t}_{0}  (t-p)^{-\mu}e^{-\delta(t-p)}  ~dp \\
\left|u(t)\right|_{D(A^{\mu)}} & \leq M_{\mu}t^{-\mu} e^{-\delta t} +M_{\mu}N \int^{\infty}_{0}  (t-p)^{-\mu}e^{-\delta(t-p)}  ~dp\leq M_{\mu}t^{-\mu} e^{-\mu t} +M_{\mu}N ~\Gamma(1-\mu) := ~ Q_{\mu} < \infty
\end{align*}

Thus, $\{ u(t); t\geq 1\}$ is bounded $ D(A^{\mu})$, where $\mu > \frac{1}{2}$, and the orbit $\{ u(t); t\geq 1\}$ is pre compact in $\mathcal{V}$

\end{proof}  
\end{theorem} 

\begin{corollary}{\label{Coro-PreCompact}}
The $\Omega$- limit set ~ $\Omega(u_{0}) = \cap_{q \geq 1} \overline{\{ u(t); t \geq q\}}$ exists and is compact in $\mathcal{V}$.
\begin{proof}
   In theorem (\ref{Pre_compact}), we have shown that $\{ u(t); t \geq q\}$ is pre-compact in $\mathcal{V}$ for $q\geq 1$. As the closure of the pre-compact set is also a pre-compact set, $\overline{\{ u(t); t \geq q\}}$ is pre-compact in $\mathcal{V}$ for $q\geq 1$. \\
    Additionally, $\overline{\{ u(t); t \geq q\}}$ is closed, and thus, is complete in $\mathcal{V}$-norm. Hence, the completion and pre-compactness of $\overline{\{ u(t); t \geq q\}}$ implies that $\overline{\{ u(t); t \geq q\}}$ is compact $ \forall q\geq 1 \in \mathcal{V}$. \\
    Hence, $\Omega(u_{0})$ is the intersection of the decreasing nonempty compact sets and is not an empty set in $\mathcal{V}$.
\end{proof}
\end{corollary}

Referencing the more detailed result of Proposition 1.3 of Jandoubi \cite{Jandoubi}, in the next theorem, we derive the version of the Lojasiewicz-Simon Inequality satisfied by the energy of the system (\ref{PB}).\\
It can be easily verified that  $\|u\|^{2n}_{L^{2n}}$ is analytic in $u$ in the following energy function $ \mathcal{Y}: \mathcal{V}\longrightarrow R $.

\begin{align}{\label{energynorm}}
    \mathcal{Y}(u) = \frac{1}{2} \|u\|^{2}_{\mathcal{V}} + \frac{1}{2n} \|u\|^{2n}_{L^{2n}}, ~~~~ n \in N 
\end{align} 
Moreover, suppose that $\mathcal{S}$ be collection of critical points of the problem (\ref{PB}), that is:

\begin{align*}
    \mathcal{S}:=\{\varphi: \|u\|^{2}_{\mathcal{V}}\varphi -\varphi, \varphi\vert_{\partial \mathcal{O}}=0 \}. 
\end{align*}

\begin{theorem}
\label{Lwoj} For a given $\varphi \in \mathcal{S} $ there exists $\theta \in \left( 0,\frac{1}{2}\right) $ and $\sigma >0$
such that
\begin{eqnarray}
\left\vert -  \Delta^{2}u+2 \Delta u  + \|  u\|^{2}_{{H}^{2}_{0}} ~u + 2\|    u\|^{2}_{{H}^{1}_{0}} ~u  +\| u\|^{2n}_{{L}^{2n}} u- u^{2n-1} \right\vert _{\mathcal{H}}
&\geq &\left\vert \mathcal{Y} (u)-\mathcal{Y} (\varphi )\right\vert ^{1-\theta } \label{LJ}
\end{eqnarray}
for all $u$, and such that $\left\Vert u-\varphi \right\Vert 
<\sigma $.
\end{theorem}

\begin{proof}
Let us start the proof by constructing the linearization of our main problem (\ref{PB}) about the equilibrium point $\varphi$. Set $u=v+ \varphi$
\begin{align*}
    \varphi_{t} =u_{t}-v_{t} 
     &= \left(-\Delta^{2} u+2\Delta u+\|  u\|^{2}_{{H}^{2}_{0}} ~u + 2\|    u\|^{2}_{{H}^{1}_{0}} ~u  +\| u\|^{2n}_{{L}^{2n}} u- u^{2n-1}\right)\\
    & ~~~~- \left(-\Delta^{2} v+2\Delta v+\|  v\|^{2}_{{H}^{2}_{0}} ~v + 2\|    v\|^{2}_{{H}^{1}_{0}} ~v  +\| v\|^{2n}_{{L}^{2n}} v- v^{2n-1}\right)
\end{align*}

Recall that  $ \mathcal{V} =H^{1}_{0} \cap H^{2}$ is a Banach space with the norm defined by $  \|u\|^{2}_{\mathcal{V}} = \|u\|^{2}_{H^{1}_{0}} +  \|u\|^{2}_{H^{2}} =\|u\|^{2}_{L^{2} (\Omega)} +  2\|\nabla u\|^{2}_{L^{2} (\Omega)} +   \|\Delta u\|^{2}_{L^{2} (\Omega)} $. So, it follows that:

\begin{align}
    \varphi_{t} &= \left(-\Delta^{2} u+2\Delta u+\|u\|^{2}_{\mathcal{V}}u -\|u\|^{2}_{L^{2}}u  +\| u\|^{2n}_{{L}^{2n}} u- u^{2n-1}\right)\notag \\ &- \left(-\Delta^{2} v+2\Delta v+\|v\|^{2}_{\mathcal{V}}v -\|v\|^{2}_{L^{2}}v +\| v\|^{2n}_{{L}^{2n}} v- v^{2n-1}\right) \notag \\
    &=-\Delta^{2} \left(u-v\right)+2\Delta  \left(u-v\right) + \left( \|u\|^{2}_{\mathcal{V}}u-\|v\|^{2}_{\mathcal{V}}v\right)-\left( \|u\|^{2}_{L^{2}}u-\|v\|^{2}_{L^{2}}v\right)\notag \\&+\left(\| u\|^{2n}_{{L}^{2n}} u-\| v\|^{2n}_{{L}^{2n}} v \right) - \left(u^{2n-1}-v^{2n-1}\right) \notag \\
    &=: -\Delta^{2} \varphi +2\Delta  \varphi + F_{1}+ F_{2}+ F_{3}+ F_{4}
\end{align}

By neglecting the square and higher powers of $\varphi$, let us calculate $F_{1}, F_{2}, F_{3}, F_{4}$.

\begin{align}{\label{F_1}}
    F_{1}= \|u\|^{2}_{\mathcal{V}}u-\|v\|^{2}_{\mathcal{V}}v&= \left( \|u\|^{2}_{\mathcal{V}}-\|v\|^{2}_{\mathcal{V}}\right)u+\|v\|^{2}_{\mathcal{V}}(u-v) \notag \\
    &=\left( \|u\|^{2}_{\mathcal{V}}-\|v\|^{2}_{\mathcal{V}}\right)u+\|v\|^{2}_{\mathcal{V}}\varphi 
      =\left( \|u\|^{2}_{\mathcal{V}}-\|v\|^{2}_{\mathcal{V}}\right)(v +\varphi)+\|v\|^{2}_{\mathcal{V}}\varphi \notag \\
      &=\left( \|v + \varphi \|^{2}_{\mathcal{V}}-\|v\|^{2}_{\mathcal{V}}\right)(v +\varphi)+\|v\|^{2}_{\mathcal{V}}\varphi\notag\\& = \left( \|v \|^{2}_{\mathcal{V}}+ \|\varphi  \|^{2}_{\mathcal{V}}+ 2 \langle v, \varphi \rangle_{\mathcal{V}} -\|v\|^{2}_{\mathcal{V}}\right)(v +\varphi)+\|v\|^{2}_{\mathcal{V}}\varphi \notag \\
      &= \left( \|\varphi  \|^{2}_{\mathcal{V}}+2 \langle v, \varphi \rangle _{\mathcal{V}}\right)(v +\varphi)+\|v\|^{2}_{\mathcal{V}}\varphi=   2 v\langle v, \varphi \rangle_{\mathcal{V}}  +\|v\|^{2}_{\mathcal{V}}\varphi
\end{align}

And,

\begin{align}{\label{F_2}}
    F_{2}= \|u\|^{2}_{L^{2}}u-\|v\|^{2}_{L^{2}}v&= \left( \|u\|^{2}_{L^{2}}-\|v\|^{2}_{L^{2}}\right)u+\|v\|^{2}_{L^{2}}(u-v) =\left( \|u\|^{2}_{L^{2}}-\|v\|^{2}_{L^{2}}\right)u+\|v\|^{2}_{L^{2}}\varphi \notag \\
      &=\left( \|u\|^{2}_{L^{2}}-\|v\|^{2}_{L^{2}}\right)(v +\varphi)+\|v\|^{2}_{L^{2}}\varphi =\left( \|v + \varphi \|^{2}_{L^{2}}-\|v\|^{2}_{L^{2}}\right)(v +\varphi)+\|v\|^{2}_{L^{2}}\varphi \notag\\
      &= \left( \|v \|^{2}_{L^{2}}+ \|\varphi  \|^{2}_{L^{2}}- 2 \langle v, \varphi \rangle_{L^{2}} -\|v\|^{2}_{L^{2}}\right)(v +\varphi)+\|v\|^{2}_{L^{2}}\varphi \notag \\
      &= \left( \|\varphi  \|^{2}_{L^{2}}+2 \langle v, \varphi \rangle_{L^{2}} \right)(v +\varphi)+\|v\|^{2}_{L^{2}}\varphi=   2 v\langle v, \varphi \rangle_{L^{2}}  +\|v\|^{2}_{L^{2}}\varphi
\end{align}

Now, 

\begin{align} {\label{F_3}}
F_{3} &= u \| u \|_{L^{2n}}^{2n} - v \| v \|_{L^{2n}}^{2n} = \left( u - v \right) \| u \|_{L^{2n}}^{2n} + \left( \| u \|_{L^{2n}}^{2n} - \| v \|_{L^{2n}}^{2n} \right) v \notag \\
&= \varphi \| v + \varphi \|_{L^{2n}}^{2n} + \left( \| v + \varphi \|_{L^{2n}}^{2n} - \| v \|_{L^{2n}}^{2n} \right) v \notag \\
&= \varphi \| v \|_{L^{2n}}^{2n} + \left( \| v \|_{L^{2n}}^{2n} + 2n \left\langle v^{2n-1}, \varphi \right\rangle - \| v \|_{L^{2n}}^{2n} \right) v \notag \\
&= \varphi \| v \|_{L^{2n}}^{2n} + 2n \left\langle v^{2n-1}, \varphi \right\rangle v  
\end{align}

Finally, $F_{4},$

\begin{align}{\label{F_4}}
F_{4} &=u^{2n-1}-v^{2n-1}=\left( v+\varphi \right) ^{2n-1}-v^{2n-1}  \notag \\
& =v^{2n-1}+(2n-1)v^{2n-2}\varphi -v^{2n-1}  
=(2n-1)v^{2n-2}\varphi.  
\end{align}

Therefore, using (\ref{F_1}),(\ref{F_2}),(\ref{F_3}), and (\ref{F_4}), we have:

\begin{align}
     \varphi_{t}
     &= -\Delta^{2} \varphi +2 \Delta  \varphi  +\left(\|v\|^{2}_{\mathcal{V}}+\|v\|^{2}_{L^{2}}+  \| v \|_{L^{2n}}^{2n}\right) \varphi
+  \left( \langle 2v, \varphi \rangle_{\mathcal{V}} + \langle 2v, \varphi \rangle_{L^{2}}  + 2n \left\langle v^{2n-1}, \varphi \right\rangle \right) v\notag\\& - (2n-1)v^{2n-2}\varphi
\end{align}

We define the linear operator $L$ on $\mathcal{V}$ by

\begin{align}
    Lw & \equiv  -\Delta^{2} w +2 \Delta  w +\left(\|\varphi\|^{2}_{\mathcal{V}}+\|\varphi\|^{2}_{L^{2}}+  \| \varphi \|_{L^{2n}}^{2n}\right) w
+  \left( \langle 2\varphi, w \rangle_{\mathcal{V}} + \langle 2\varphi, w \rangle_{L^{2}}  + 2n \left\langle \varphi^{2n-1}, w \right\rangle \right) \varphi \notag\\& - (2n-1)\varphi^{2n-2}w=0 \notag \\
w|_{\Gamma } &=0
\end{align}

Clearly $L$ is a self-adjoint operator. From the  linear
elliptic boundary problems we can deduce that there is a  positive $\lambda $ such that $
\lambda I+L$ is invertible.Thus, we have Fredholm alternative result for
the problem,
\begin{eqnarray}
Lw &=&h,\text{ }x\in D  \label{Lj6} \\
w|_{\Gamma } &=&0.  \notag
\end{eqnarray}

Using Fredholm alternative, we can infer that  if $\ker L=\phi ,$ then for any $
h\in \mathcal{H},$ the problem (\ref{Lj6}) has a unique solution $w\in \mathcal{V}.$ If $
\ker L\neq \phi $ and have finite dimension so in this case problem (\ref
{Lj6}) have a solution iff $h\in \left( \ker L\right) ^{\bot }.$ Let
$\left\{ e_{j}\right\} _{j=1}^{n}$ is the normalized orthogonal basis of $\ker L$ in $\mathcal{H}$ and $\Pi $ be the projection from $\mathcal{H}$ onto $\ker L.$ Define an operator $\mathcal{L} :\mathcal{V}\rightarrow \mathcal{H}$ by
\begin{equation}
\mathcal{L} w=Lw+\Pi w  \label{Lj7}
\end{equation}

The operator $\mathcal{L}$ is a bijection. For $u= v+ \varphi$, we define operator $M: \mathcal{V} \rightarrow \mathcal{H}$ by:

\begin{align}{\label{M-0}}
    M(v)\equiv -\Delta^{2} u+2\Delta u+\|u\|^{2}_{\mathcal{V}}u -\|u\|^{2}_{L^{2}}u  +\| u\|^{2n}_{{L}^{2n}} u- u^{2n-1}
\end{align}

We argue that $DM(0)=L$. To do so, it is sufficient to prove $\frac{\left\vert M(h)-M(0)-L\right\vert }{
\left\vert h\right\vert _{\mathcal{H}}} \rightarrow 0$ as $\left\vert h\right\vert _{\mathcal{H}}\rightarrow 0$. Where $D$ is Frechet derivative.\\
From the definition of $M$, we have:

\begin{align*}
    M(0)= -\Delta^{2} \varphi +2\Delta  \varphi+\|   \varphi\|^{2}_{\mathcal{V}} ~ \varphi -\|     \varphi\|^{2}_{L^{2}} ~ \varphi  +\|  \varphi\|^{2n}_{{L}^{2n}}  \varphi-  \varphi^{2n-1}
\end{align*}

and 
\begin{align}{\label{M-h}}
    M(h) &= -\Delta^{2} (h+\varphi)+2\Delta (h+\varphi)+\|  h+\varphi\|^{2}_{\mathcal{V}} ~(h+\varphi) -\|    h+\varphi\|^{2}_{L^{2}} ~u  +\| h+\varphi\|^{2n}_{{L}^{2n}} (h+\varphi)- (h+\varphi)^{2n-1} \notag\\
    &=-\Delta^{2} (h+\varphi)+2\Delta (h+\varphi)+ \left( \|h \|^{2}_{\mathcal{V}}+ \|\varphi  \|^{2}_{\mathcal{V}}+ 2 \langle h, \varphi \rangle_{\mathcal{V}}\right)(h+\varphi) - \left( \|h \|^{2}_{L^{2}}+ \|\varphi  \|^{2}_{L^{2}}+ 2 \langle h, \varphi \rangle_{L^{2}}\right)(h+\varphi) \notag\\
    &~~+\left(\| \varphi\|^{2n}_{{L}^{2n}} +2n \langle\varphi^{2n-1}, h \rangle+...  \right) (h+\varphi) -\left(\varphi ^{2n-1}+(2n-1)\varphi ^{2n-2}h+...+h^{2n-1}\right)
\end{align}

Using (\ref{M-0}) and (\ref{M-h}), we have:

\begin{eqnarray*}
&&\left\vert
\begin{array}{c}
 M(h)-M(0) +\Delta^{2} h -2 \Delta  h -\left(\|\varphi\|^{2}_{\mathcal{V}}+\|\varphi\|^{2}_{L^{2}}+  
 \| \varphi \|_{L^{2n}}^{2n}\right) h \\
-  \left( \langle 2\varphi, h \rangle_{\mathcal{V}} + 
\langle 2\varphi, h \rangle_{L^{2}}  - 2n \left\langle \varphi^{2n-1}, h \right\rangle \right) \varphi - (2n-1)\varphi^{2n-2}h
\end{array}
\right\vert \\
&=&o\left( \left\vert h\right\vert _{\mathcal{H}}\right) \rightarrow 0\text{ as }
\left\vert h\right\vert _{\mathcal{H}}\rightarrow 0
\end{eqnarray*}

Therefore, we have $DM(0)=L$. Set
\begin{equation*}
Tv=M(v)+\Pi v.
\end{equation*}
Then
\begin{equation*}
DT\left( 0\right) =\mathcal{L}.
\end{equation*}

The bijection of  $\mathcal{L} $ implies that $DN\left( 0\right)$ is bijective too. So,  by
local inversion theorem, there is a  neighbourhood $\mathcal{O}_{\mathcal{V}}$ of origin in $\mathcal{V}$
and neighbourhood $\mathcal{O}_{\mathcal{H}}$ of origin in $\mathcal{H},$ together with with inverse $\Upsilon
:O_{\mathcal{H}}\rightarrow O_{\mathcal{V}}$ such that
\begin{eqnarray}
T\left( \Upsilon \left( g\right) \right) &=&g,\text{ for all }g\in \mathcal{O}_{\mathcal{H}},
\label{Lj10} \\
\text{and }\Upsilon \left( T\left( v\right) \right) &=&v,\text{ for all }v\in
\mathcal{O}_{\mathcal{V}}.  \label{Lj11}
\end{eqnarray}

As $F(u)=\|  u\|^{2}_{{H}^{2}_{0}} ~u + 2\|    u\|^{2}_{{H}^{1}_{0}} ~u  +\| u\|^{2n}_{{L}^{2n}} u- u^{2n-1}$ is analytic in $u$, so is $M$. Since $T=M +\Pi$ is analytic, so this implies that $\Upsilon$ is analytic too. In addition, there is a constant $K>0$ such that:

\begin{eqnarray}
\left\Vert \Upsilon \left( g_{1}\right) -\Upsilon\left( g_{2}\right) \right\Vert_{\mathcal{V}}
&\leq &K\left\vert g_{1}-g_{2}\right\vert _{\mathcal{H}},\text{ for all }
g_{1},g_{2}\in \mathcal{O}_{\mathcal{H}}  \label{Lj12} \\
\left\vert T\left( v_{1}\right) -T\left( v_{2}\right) \right\vert _{\mathcal{H}} &\leq
&K\left\Vert v_{1}-v_{2}\right\Vert_{\mathcal{V}} ,\text{ for all }v_{1},v_{2}\in \mathcal{O}_{\mathcal{V}}.
\label{Lj13}
\end{eqnarray}

Here, we take $K$ as an arbitrary constant. Assume that  $\xi =\left( \xi _{k}\right) _{k=1}^{d}\in\mathbb{R}^{d}$ and since $\Pi v\in \ker L$ so $\Pi v=\sum_{k=1}^{d}\xi _{k}e_{k}.$ We
can construct $\xi $ sufficiently small to have $\Pi v=\sum_{k=1}^{d}\xi
_{k}e_{k}\in \mathcal{O}_{\mathcal{H}}.$

Define $\Xi:\mathbb{R}^{d}\rightarrow\mathbb{R}$ as
\begin{equation*}
\Xi \left( \xi \right) =\mathcal{Y} \left( \Upsilon \left( \Pi v\right) +\varphi
\right) :=\mathcal{Y} \left( \gamma \left( v\right) \right) .
\end{equation*}

It is obvious that  $\Xi \left( \xi \right) $ is analytic in a small neighbourhood of origin in $\mathbb{R}^{d}.$ And for $v\in O_{\mathcal{V}}$, the Frechet derivative $DT\left( v\right):\mathcal{V}\rightarrow \mathcal{H}$ is continuous linear transformation. Similarly $g\in O_{\mathcal{H}}$, $D\Upsilon \left( g\right) :\mathcal{H}\rightarrow \mathcal{V}$ is also continuous linear
transformation. Differentiate $\Xi $ w.r.t $\xi _{k},$, it follows that:
\begin{equation}
\frac{\partial \Xi }{\partial \xi _{k}}=D\mathcal{Y} \left( \gamma \left(
v\right) \right) \circ \left( D\gamma \left( v\right) \right) =\left( D\mathcal{Y}
\left( \gamma \left( v\right) \right) \right) \left( D\gamma \left( v\right)
\right)  \label{Lj15}
\end{equation}

Using Frëchet derivative of $\Xi $  with $\alpha_{1}
=\gamma \left( v\right) $ and $\alpha_{2} =D\gamma \left( v\right) ,$ we can deduce that:

\begin{eqnarray*}
\frac{\partial \Xi }{\partial \xi _{k}} &=&\left( D\mathcal{Y} \left( \alpha_{1}
\right) \right) \left( \alpha_{2} \right) =\left\langle \Delta^2 \alpha_1 - 2\Delta \alpha_1 +\alpha_1 + \alpha_1^{2n-1},\alpha_{2} \right\rangle \\
&=&-\left\langle- \Delta^2 \alpha_1 +2\Delta \alpha_1  - \alpha_1^{2n-1},\alpha_{2} \right\rangle + \left\langle\alpha_1, \alpha_2\right\rangle\\
&=& - \left\langle -\Delta^{2} \alpha_{1} +2\Delta \alpha_{1} -\alpha_{1}
^{2n-1} +\left(\|  \alpha_{1}\|^{2}_{{H}^{2}_{0}}  + 2\|    \alpha_{1}\|^{2}_{{H}^{1}_{0}}   +\| \alpha_{1}\|^{2n}_{{L}^{2n}} \right) \alpha_{1},\alpha_{2} \right\rangle \\
& &+ \left(1+\|  \alpha_{1}\|^{2}_{{H}^{2}_{0}}  + 2\|    \alpha_{1}\|^{2}_{{H}^{1}_{0}}   +\| \alpha_{1}\|^{2n}_{{L}^{2n}} \right) \left \langle \alpha_{1}, \alpha_{2} \right\rangle \\
&=&  -\left\langle  \alpha'_{1},\alpha_{2} \right\rangle+ \left(1+\|  \alpha_{1}\|^{2}_{{H}^{2}_{0}}  + 2\|    \alpha_{1}\|^{2}_{{H}^{1}_{0}}   +\| \alpha_{1}\|^{2n}_{{L}^{2n}} \right) \left \langle \alpha_{1}, \alpha_{2} \right\rangle \\
&=& -\left\langle  \left(\Upsilon \left( \Pi v\right) +\varphi\right)',\alpha_{2} \right\rangle+ \left(1+\|  \alpha_{1}\|^{2}_{{H}^{2}_{0}}  + 2\|    \alpha_{1}\|^{2}_{{H}^{1}_{0}}   +\| \alpha_{1}\|^{2n}_{{L}^{2n}} \right) \left \langle \alpha_{1}, \alpha_{2} \right\rangle \\
&=&- \left\langle  M  \left(\Pi v\right) ,\alpha_{2} \right\rangle+ \left(1+\|  \alpha_{1}\|^{2}_{{H}^{2}_{0}}  + 2\|    \alpha_{1}\|^{2}_{{H}^{1}_{0}}   +\| \alpha_{1}\|^{2n}_{{L}^{2n}} \right) \left \langle \alpha_{1}, \alpha_{2} \right\rangle
\end{eqnarray*}

Applying the absolute value and the Cauchy–Schwarz inequality, it follows that:

\begin{align*}
    \left |{\frac{\partial \Xi }{\partial \xi _{k}}}\right| &= \left |\left\langle  M  \left(\Pi v\right) ,\alpha_{2} \right\rangle- \left(1+\|  \alpha_{1}\|^{2}_{{H}^{2}_{0}}  + 2\|    \alpha_{1}\|^{2}_{{H}^{1}_{0}}   +\| \alpha_{1}\|^{2n}_{{L}^{2n}} \right) \left \langle \alpha_{1}, \alpha_{2} \right\rangle\right|\\
    &\leq \left| \left\langle  M  \left(\Pi v\right) ,\alpha_{2} \right\rangle \right| + \left(1+\|  \alpha_{1}\|^{2}_{{H}^{2}_{0}}  + 2\|    \alpha_{1}\|^{2}_{{H}^{1}_{0}}   +\| \alpha_{1}\|^{2n}_{{L}^{2n}} \right)\left|\left \langle \alpha_{1}, \alpha_{2} \right\rangle\right|\\
    & \leq \left|M  \left(\Pi v\right)\right|_{\mathcal{H}}\left|\alpha_{2}\right|_{\mathcal{H}} +\left(1+\|  \alpha_{1}\|^{2}_{{H}^{2}_{0}}  + 2\|    \alpha_{1}\|^{2}_{{H}^{1}_{0}}   +\| \alpha_{1}\|^{2n}_{{L}^{2n}} \right) \left|\alpha_{1}\right|_{\mathcal{H}}\left|\alpha_{2}\right|_{\mathcal{H}} \\
    &= \left[ \left|M  \left(\Pi v\right)\right|_{\mathcal{H}} +\left(1+\|  \alpha_{1}\|^{2}_{{H}^{2}_{0}}  + 2\|    \alpha_{1}\|^{2}_{{H}^{1}_{0}}   +\| \alpha_{1}\|^{2n}_{{L}^{2n}} \right) \left|\alpha_{1}\right|_{\mathcal{H}} \right]\left|\alpha_{2}\right|_{\mathcal{H}}
\end{align*}

As $v\in O_{\mathcal{V}}$, $\Pi v\in O_{H\text{ }}$ and $\Upsilon\left( \Pi v\right) \in
O_{\mathcal{V}}$ so

\begin{eqnarray}
\left\vert \alpha_1 \right\vert _{\mathcal{H}} &=&\left\vert \Upsilon \left( \Pi v\right)
+\varphi \right\vert _{\mathcal{H}}\leq \left\vert \Upsilon \left( \Pi v\right)
\right\vert _{\mathcal{H}}+\left\vert \varphi \right\vert _{\mathcal{H}}  \notag \\
&\leq &k\left\Vert \Upsilon \left( \Pi v\right) \right\Vert_{\mathcal{V}} +\left\Vert \varphi
\right\Vert_{\mathcal{V}} :=C_{1}<\infty ,  \label{C1}
\end{eqnarray}

and,

\begin{eqnarray}
1+\left\Vert \alpha_1 \right\Vert^{2} _{H^{2}_{0}} &=&1+\left\Vert \Upsilon \left( \Pi v\right)
+\varphi \right\Vert^{2} _{H^{2}_{0}}\leq 1+\left\Vert \Upsilon \left( \Pi v\right)
\right\Vert ^{2}_{H^{2}_{0}}+\left\Vert \varphi \right\Vert^{2}_{H^{2}_{0}}  \notag \\
&\leq &1+k^2 \left(\left\Vert \Upsilon \left( \Pi v\right) \right\Vert_{\mathcal{V}}^{2} +\left\Vert \varphi
\right\Vert_{\mathcal{V}}^{2}\right) :=C_{2}<\infty ,  \label{C2}
\end{eqnarray}

Also,

\begin{eqnarray}
2\left\Vert \alpha_1 \right\Vert^{2} _{H^{1}_{0}} &=&2\left\Vert \Upsilon \left( \Pi v\right)
+\varphi \right\Vert^{2} _{H^{1}_{0}} \leq 2\left\Vert \Upsilon \left( \Pi v\right)
\right\Vert ^{2}_{H^{1}_{0}}+2\left\Vert \varphi \right\Vert^{2}_{H^{1}_{0}}  \notag \\
&\leq &2 k^2 \left(\left\Vert \Upsilon \left( \Pi v\right) \right\Vert_{\mathcal{V}}^{2} +\left\Vert \varphi
\right\Vert_{\mathcal{V}}^{2}\right) :=C_{3}<\infty ,  \label{C3}
\end{eqnarray}

Moreover,

\begin{eqnarray}
\left\vert \alpha_1 \right\vert _{{L}^{2n}}^{2n} &\leq &k^{2n}\left\Vert \alpha_1
\right\Vert ^{2n}=k^{2n}\left\Vert \Upsilon \left( \Pi v\right) +\varphi
\right\Vert_{\mathcal{V} }^{2n}\leq k^{2n}\left( \left\Vert \Upsilon \left( \Pi v\right) \right\Vert_{\mathcal{V} }
+\left\Vert \varphi \right\Vert_{\mathcal{V} } \right) ^{2n}:=C_{3}<\infty .  \label{C3}
\end{eqnarray}

Now, for $\left|\alpha_{2}\right|$, we have:

\begin{eqnarray*}
\alpha_{2} &=&D\gamma \left( v\right) =D\left( \Upsilon \left( \Pi v\right) +\varphi
\right) =D\left( \Upsilon \left( \Pi v\right) \right) =D\left( \Upsilon \left( \Pi v\right) \right) \circ \frac{\partial }{\partial
\xi _{k}}\left( \Pi v\right)=\left( D\Upsilon \left( \Pi v\right) \right) \left( e_{k}\right)
\end{eqnarray*}
Applying $\mathcal{H}-$norm on both sides
\begin{eqnarray}
\left\vert \alpha_{2} \right\vert _{\mathcal{H}} &=&\left\vert \left( D\Upsilon \left( \Pi
v\right) \right) \left( e_{k}\right) \right\vert _{\mathcal{H}}\leq \left\vert D\Upsilon \left( \Pi v\right) \right\vert _{L\left( H,V\right)
}\left\vert e_{k}\right\vert _{\mathcal{H}}=C_{4}<\infty.  \label{C4}
\end{eqnarray}

Thus, using inequalities (\ref{C1}), (\ref{C2}), (\ref{C3}), and (\ref{C4}) we have:

\begin{eqnarray}
\left\vert \frac{\partial \Xi }{\partial \xi _{k}}\right\vert &\leq &
\left[ \left\vert M\left( \Pi v\right) \right\vert _{\mathcal{H}}+\left(
C_{2}+C_{3}\right) C_{1}\right] C_{4}  \leq
C\left\vert M\left( \Pi v\right) \right\vert _{\mathcal{H}}.  \label{LK}
\end{eqnarray}

Now, $\varphi $ is equilibrium point so  from
equation (\ref{Lj15}) we can deduce that $\xi =0$ is critical point of $\Gamma \left( \xi
\right) .$ Using inequality (\ref{LK}), we have
\begin{eqnarray}
\left\vert \nabla \Xi \left( \xi \right) \right\vert &\leq &K\left\vert
M\left( \Upsilon \left( \Pi v\right) \right) \right\vert _{\mathcal{H}}=K\left\vert M\left( \Upsilon \left( \Pi v\right) \right)
-M(v)+M(v)\right\vert _{\mathcal{H}}  \notag \\
&\leq &K\left\vert M(v)\right\vert _{\mathcal{H}}+K\left\vert M\left( \Upsilon \left( \Pi
v\right) \right) -M(v)\right\vert _{\mathcal{H}}  \label{Lj16}
\end{eqnarray}

As $T(v)=M(v)+\Pi v$ then from inequality (\ref{Lj13}) it follows that:
\begin{eqnarray}
\left\vert M\left( v_{1}\right) -M\left( v_{2}\right) \right\vert _{\mathcal{H}} &\leq
&\left\vert \Pi \left( v_{1}\right) -\Pi \left( v_{2}\right) \right\vert
_{\mathcal{H}}+\left\vert T\left( v_{1}\right) -T\left( v_{2}\right) \right\vert _{\mathcal{H}}
\leq \left\vert v_{1}-v_{2}\right\vert _{\mathcal{H}}+K\left\Vert
v_{1}-v_{2}\right\Vert_{\mathcal{V}}  \notag \\
&\leq &K\left\Vert v_{1}-v_{2}\right\Vert_{\mathcal{V}} +K\left\Vert v_{1}-v_{2}\right\Vert_{\mathcal{V}}
\leq K\left\Vert v_{1}-v_{2}\right\Vert_{\mathcal{V}}  \label{Lj17}
\end{eqnarray}

Furthermore,
\begin{equation}
\left\Vert \Upsilon \left( \Pi v\right) -v\right\Vert =\left\Vert \Upsilon \left(
\Pi v\right) -\Upsilon \left( T(v)\right) \right\Vert \leq K\left\vert \Pi
v-T(v)\right\vert _{\mathcal{H}}=K\left\vert M(v)\right\vert _{\mathcal{H}}  \label{Lj18}
\end{equation}
Using inequalities (\ref{Lj16}),(\ref{Lj17}) and (\ref{Lj18}), it follows that:
\begin{eqnarray}
\left\vert \nabla \Xi \left( \xi \right) \right\vert &\leq &K\left\vert
M(v)\right\vert _{\mathcal{H}}+K\left\Vert \Upsilon \left( \Pi v\right) -v\right\Vert
\leq K\left\vert M(v)\right\vert _{\mathcal{H}}+K\left\vert M(v)\right\vert _{\mathcal{H}}
\leq K\left\vert M(v)\right\vert _{\mathcal{H}}  \label{Lj19}
\end{eqnarray}

For  $t\in \lbrack 0,1],$ and $v\in O_{\mathcal{V}},$ $v+t\left( \Upsilon \left( \Pi
v\right) -v\right) \in O_{\mathcal{V}}.$ If $u=v+\varphi ,$ then using inequality (\ref
{Lj17}) and (\ref{Lj18}) we can infer that
\begin{eqnarray*}
\left\vert \mathcal{Y} \left( u\right) -\Xi \left( \xi \right) \right\vert
&=&\left\vert \mathcal{Y} \left( u\right) -\mathcal{Y} \left( \Upsilon \left( \Pi v\right)
+\varphi \right) \right\vert=\left\vert \int_{0}^{1}\frac{d}{dt}\mathcal{Y} \left( tv+(1-t)\Upsilon \left( \Pi
v\right) +\varphi \right) dt\right\vert
\end{eqnarray*}

Now, calculate $\frac{d}{dt}\mathcal{Y} \left( tv+(1-t)\Upsilon \left( \Pi v\right)
+\varphi \right) $ in the following step by setting $\zeta :=tv+(1-t)\Upsilon \left( \Pi v\right) +\varphi $
then
\begin{eqnarray*}
\mathcal{Y} \left( tv+(1-t)\Upsilon \left( \Pi v\right) +\varphi \right) &=&\mathcal{Y} \left(
\zeta \right)  = \frac{1}{2}
\left\Vert   \zeta \right\Vert_{\mathcal{V}} ^{2} +\frac{1}{2n}
\left\Vert \zeta \right\Vert_{{L^{2n}}} ^{2n}\\
\end{eqnarray*}

As $\frac{d\zeta }{dt} :=v-\Upsilon \left( \Pi v\right)$. Taking derivative  w.r.t $t$, and using integration by parts, we have

\begin{eqnarray*}
\frac{d}{dt}\mathcal{Y} \left( \zeta \right) &=&\frac{1}{2}\frac{d}{dt}
\int_{D}\left\Vert  \Delta \zeta \right\Vert_{L^{2}(\mathcal{O})} ^{2}dx+ \frac{d}{dt}+\int_{D}\left\Vert  \nabla \zeta \right\Vert_{L^{2}(\mathcal{O})} ^{2}dx+\frac{1}{2}
\int_{D}\left\Vert   \zeta \right\Vert_{L^{2}(\mathcal{O})} ^{2}dx+\frac{1}{2n}\frac{d}{dt}
\int_{D}\left\Vert \zeta \right\Vert_{{L^{2n}}} ^{2n}dx\\
&=&\int_{D}   \Delta \zeta\cdot \Delta \frac{d\zeta }{dt}dx+2\int_{D}   \Delta \zeta\frac{d\zeta }{dt}dx+\int_{D}  \zeta\frac{d\zeta }{dt}dx + \int_{D}\zeta ^{2n-1}\frac{
d\zeta }{dt}dx \\  
&=&-\left\langle \Delta^{2} \zeta +2 \Delta \zeta+\zeta - \zeta ^{2n-1},\frac{d\zeta }{dt}\right\rangle \\
&=&- \left\langle -\Delta^{2} \zeta + 2\Delta \zeta - \zeta^{2n-1} + \left( \| \zeta \|^{2}_{H^{2}_{0}} + 2 \| \zeta \|^{2}_{H^{1}_{0}} + \| \zeta \|^{2n}_{L^{2n}} \right) \zeta, \frac{d\zeta}{dt} \right\rangle \\
& & + \left(1+  \| \zeta \|^{2}_{H^{2}_{0}} + 2 \| \zeta \|^{2}_{H^{1}_{0}} + \| \zeta \|^{2n}_{L^{2n}} \right) \left\langle \zeta, \frac{d\zeta}{dt} \right\rangle
\\
\left\vert \frac{d}{dt}\mathcal{Y} \left( \zeta \right) \right\vert  &\leq& \left\vert \left\langle -\Delta^{2} \zeta + 2\Delta \zeta - \zeta^{2n-1} + \left( \| \zeta \|^{2}_{H^{2}_{0}} + 2 \| \zeta \|^{2}_{H^{1}_{0}} + \| \zeta \|^{2n}_{L^{2n}} \right) \zeta, \frac{d\zeta}{dt} \right\rangle \right\vert \\
& & +  \left(1+  \| \zeta \|^{2}_{H^{2}_{0}} + 2 \| \zeta \|^{2}_{H^{1}_{0}} + \| \zeta \|^{2n}_{L^{2n}}+\|\zeta\|_{L^{2}}^{2} \right) \left\vert \left\langle \zeta, \frac{d\zeta}{dt} \right\rangle \right\vert
\\
&=& \left\vert \left\langle -\Delta^{2} \zeta + 2\Delta \zeta - \zeta^{2n-1} + \left( \| \zeta \|^{2}_{H^{2}_{0}} + 2 \| \zeta \|^{2}_{H^{1}_{0}} + \| \zeta \|^{2n}_{L^{2n}} \right) \zeta, \frac{d\zeta}{dt} \right\rangle \right\vert + \left(1+  \| \zeta \|^{2}_{\mathcal{V}}  + \| \zeta \|^{2n}_{L^{2n}} \right) \left\vert\left\langle \zeta, \frac{d\zeta}{dt} \right\rangle\right\vert
 \end{eqnarray*}

Now, we will write the expression $ 1+  \| \zeta \|^{2}_{\mathcal{V}}  + \| \zeta \|^{2n}_{L^{2n}}  $ interms of energy functional $\mathcal{Y}(\zeta).$\\

From  the definition of energy functional $\mathcal{Y}(\zeta)$, we get

\begin{align*}
    \| \zeta \|_{\mathcal{V}}^{2} = 2 \left( \mathcal{Y}(\zeta) - \frac{1}{2n} \| \zeta \|_{L^{2n}}^{2n} \right), \quad and \quad  \quad \| \zeta \|_{L^{2n}}^{2n} \leq 2n \mathcal{Y}(\zeta)
\end{align*}
And,
\begin{align*}
    1 + \| \zeta \|_{\mathcal{V}}^{2} + \| \zeta \|_{L^{2n}}^{2n} &= 1 + 2 \left( \mathcal{Y}(\zeta) - \frac{1}{2n} \| \zeta \|_{L^{2n}}^{2n} \right) + \| \zeta \|_{L^{2n}}^{2n} =1 + 2 \mathcal{Y}(\zeta) + \left(1 - \frac{1}{n}\right) \| \zeta \|_{L^{2n}}^{2n} \\
    &\leq  1 + 2n \mathcal{Y}(\zeta), \qquad   n > \frac{1}{2} 
    \end{align*}

Therefore,
\begin{eqnarray*}
\left\vert \frac{d}{dt}\mathcal{Y} \left( \zeta \right) \right\vert &\leq
&  \left\vert \left\langle M\left( tv+(1-t)\Upsilon \left( \Pi v\right) \right) ,v-\Pi
v\right\rangle \right\vert +\left(1+2n\mathcal{Y} \left( \zeta \right)\right)\left\vert \left\langle \zeta ,v-\Upsilon \left(
\Pi v\right) \right\rangle\right\vert\\
&\leq & \left\vert M\left( tv+(1-t)\Upsilon \left( \Pi v\right) \right) \right\vert
_{\mathcal{H}}\left\vert v-\Pi v\right\vert _{\mathcal{H}}+\left(1+2n\mathcal{Y} \left( \zeta \right)\right)  \left\vert
\zeta \right\vert _{\mathcal{H}}\left\vert v-\Upsilon \left( \Pi v\right) \right\vert _{\mathcal{H}}
\\
&=&\left( \left\vert M\left( \left( tv+\Upsilon \left( \Pi v\right) \right)
-t\Upsilon \left( \Pi v\right) \right) \right\vert _{\mathcal{H}}+\left(1+2n\mathcal{Y} \left( \zeta \right)\right)
\left\vert \zeta \right\vert _{\mathcal{H}}\right) \left\vert v-\Upsilon \left( \Pi
v\right) \right\vert _{\mathcal{H}} \\
&\leq &\left( \left\vert tv+\Upsilon \left( \Pi v\right) -t\Upsilon \left( \Pi
v\right) \right\vert _{\mathcal{E}}+\left(1+2n\mathcal{Y} \left( \zeta \right)\right) \left\vert \zeta
\right\vert _{\mathcal{H}}\right) \left\vert v-\Upsilon \left( \Pi v\right) \right\vert
_{\mathcal{H}} \\
&\leq &\left( \left\vert \Upsilon \left( \Pi v\right) +t(v-\Psi \left( \Pi
v\right) )\right\vert _{\mathcal{E}}+\left(1+2n\mathcal{Y} \left( \zeta \right)\right) \left\vert \zeta
\right\vert _{\mathcal{H}}\right) \left\vert v-\Upsilon \left( \Pi v\right) \right\vert
_{\mathcal{H}} \\
\int_{0}^{1}\left\vert \frac{d}{dt}\mathcal{Y} \left( \zeta \right) \right\vert dt
&\leq &\int_{0}^{1}\left( \left\vert \Upsilon \left( \Pi v\right) \right\vert
_{\mathcal{E}}+t\left\vert (v-\Upsilon \left( \Pi v\right) )\right\vert _{\mathcal{E}}+\left(1+2n\mathcal{Y} \left( \zeta \right)\right) \left\vert \zeta \right\vert _{\mathcal{H}}\right) \left\vert v-\Psi
\left( \Pi v\right) \right\vert _{\mathcal{H}}dt \\
&=&\left( \left\vert \Upsilon \left( \Pi v\right) \right\vert _{\mathcal{E}}+\left\vert
(v-\Upsilon \left( \Pi v\right) )\right\vert _{\mathcal{E}}+\underset{0\leq t\leq 1}{\max }
\left(1+2n\mathcal{Y} \left( \zeta \right)\right) \left\vert \eta \right\vert _{\mathcal{H}}\right) \left\vert
v-\Upsilon \left( \Pi v\right) \right\vert _{\mathcal{H}}
\end{eqnarray*}

Now, it is obvious that the first and third term of the above inequality are bounded by some constant $K'>0$. It follows that:

\begin{equation*}
\left\vert \mathcal{Y} \left( u\right) -\Xi \left( \xi \right) \right\vert \leq
\left( \left\vert (v-\Upsilon \left( \Pi v\right) )\right\vert _{\mathcal{E}}+K^{\prime
}\right) \left\vert v-\Upsilon \left( \Pi v\right) \right\vert _{\mathcal{H}}
\end{equation*}
From the continuous embedding $\mathcal{E}\hookrightarrow H$ and 
$\left\vert v-\Upsilon \left( \Pi v\right)
\right\vert _{\mathcal{E}}\leq K\left\vert M\left( v\right) \right\vert _{\mathcal{H}}$, we can infer that
\begin{eqnarray}
\left\vert \mathcal{Y} \left( u\right) -\Xi\left( \xi \right) \right\vert \leq
\left( K\left\vert M\left( v\right) \right\vert _{\mathcal{H}}+K^{\prime }\right)
K\left\vert M\left( v\right) \right\vert _{\mathcal{H}}\leq K\left\vert M\left( v\right) \right\vert _{\mathcal{H}}^{2}+K^{\prime
}K\left\vert M\left( v\right) \right\vert _{\mathcal{H}}\leq K\left\vert M\left( v\right) \right\vert _{\mathcal{H}}^{2}.  \label{Lj20}
\end{eqnarray}

By the Lojasiewicz inequality, there exists a small constant $\sigma >0$ and $
\theta \in \left( 0,\frac{1}{2}\right) $ such that
\begin{equation}
\left\vert \nabla \Xi \left( \xi \right) \right\vert \geq \left\vert
\Xi \left( \xi \right) -\Xi \left( 0\right) \right\vert ^{1-\theta }.
\label{Lj21}
\end{equation}
Hence, we can choose $\sigma $ sufficiently small so that $\left\vert \xi
\right\vert \leq \sigma ,\Pi v=\sum_{j=1}^{d}\xi _{k}e_{k}\in O_{\mathcal{H}}.$ Using
definition of $\Xi$ and last inequality, we get
\begin{equation}
\left\vert \nabla \Xi \left( \xi \right) \right\vert \geq \left\vert
\Xi \left( \xi \right) -\Xi \left( 0\right) \right\vert ^{1-\theta
}=\left\vert \Xi \left( \xi \right) -\mathcal{Y} (\varphi )\right\vert
^{1-\theta }.  \label{Lj22}
\end{equation}
Recall the follwoing inequality for real numbers $p$ and $q$
\begin{equation*}
\left\vert p+q\right\vert ^{1-\theta }\leq 2\left\vert p\right\vert
^{1-\theta }+\left\vert q\right\vert ^{1-\theta }
\end{equation*}
Applying this inequality by choosing  $p=\Xi \left( \xi \right) -\mathcal{Y} (\varphi )$ and $
q=\mathcal{Y} (u)-\Xi \left( \xi \right) ,$ along with inequalities (\ref{Lj19}), (
\ref{Lj20}) and (\ref{Lj22}), we have
\begin{eqnarray*}
K\left\vert M(v)\right\vert _{\mathcal{H}} &\geq &\left\vert \nabla \Xi \left( \xi
\right) \right\vert \geq \left\vert \Xi \left( \xi \right) -\mathcal{Y} (\varphi )\right\vert
^{1-\theta } \geq \frac{1}{2}\left\vert \Xi \left( \xi \right) -\mathcal{Y} (\varphi )+\mathcal{Y}
(u)-\Xi \left( \xi \right) \right\vert ^{1-\theta } \\
&&-\frac{1}{2}\left\vert \mathcal{Y} (u)-\Xi \left( \xi \right) \right\vert
^{1-\theta }=\frac{1}{2}\left\vert \mathcal{Y} (u)-\mathcal{Y} (\varphi )\right\vert ^{1-\theta }-
\frac{1}{2}K^{1-\theta }\left\vert M(v)\right\vert _{\mathcal{H}}^{2(1-\theta )}.
\end{eqnarray*}
Since $\theta <\frac{1}{2},$ so $2\left( 1-\theta \right) >1.$ Then using
last inequality we get
\begin{equation}
\left\vert M(v)\right\vert _{\mathcal{H}}\geq K\left\vert \mathcal{Y} (u)-\mathcal{Y} (\varphi
)\right\vert ^{1-\theta }  \label{Lj23}
\end{equation}

Assume that $\varepsilon >0$ is sufficiently small. Choose $\sigma $ so that
$\left\Vert v\right\Vert \leq \sigma $ and
\begin{equation*}
K\left\vert \mathcal{Y} (u)-\mathcal{Y} (\varphi )\right\vert ^{-\varepsilon }\geq 1
\end{equation*}
Using last inequality with inequality (\ref{Lj23}) and  $   M(v)= -\Delta^{2} u+2\Delta u+\|u\|^{2}_{\mathcal{V}}u -\|u\|^{2}_{L^{2}}u  +\| u\|^{2n}_{{L}^{2n}} u- u^{2n-1}:\mathcal{V}\rightarrow \mathcal{H},$ we can infer that
\begin{equation*}
\left\vert    -\Delta^{2} u+2\Delta u+\|u\|^{2}_{\mathcal{V}}u -\|u\|^{2}_{L^{2}}u  +\| u\|^{2n}_{{L}^{2n}} u- u^{2n-1}\right\vert _{\mathcal{H}}\geq \left\vert \mathcal{Y} (u)-\mathcal{Y} (\varphi
)\right\vert ^{1-\theta ^{\prime }}
\end{equation*}
where $0<\theta ^{\prime }=\theta -\varepsilon <\frac{1}{2}.$ This is the end of the proof.
\end{proof}

In the next theorem, we will prove the convergence of the solution to the problem (\ref{PB}) to an equilibrium $\Omega(u_{0})$.

\begin{theorem}\label{3.3}
If the initial data $u_{0}$ is in $ \mathcal{V}\cap\mathcal{M},$ then the unique solution of our main problem \eqref{PB} has a limit $u^{\infty } \in \mathcal{V}$ as time $t\rightarrow \infty .$
\end{theorem}

\begin{proof}

From the theorem (\ref{der}), we have

\begin{eqnarray}
\frac{d}{dt}\left( \mathcal{Y} (u\right) ) =-\left\vert u_{t}^{}\right\vert _{\mathcal{H}}^{2}.
\end{eqnarray}

Also, 

\begin{equation}
\frac{d}{dt}\left( \mathcal{Y} (u(t)\right) ^{\theta }=-\frac{\theta \left\vert
u_{t}^{}\right\vert _{\mathcal{H}}^{2}}{\mathcal{Y} (u)^{1-\theta }}=-\frac{\theta
\left\vert -\Delta^{2} u+2\Delta u+\|u\|^{2}_{\mathcal{V}}u -\|u\|^{2}_{L^{2}}u  +\| u\|^{2n}_{{L}^{2n}} u- u^{2n-1}\right\vert
_{\mathcal{H}}^{2}}{\mathcal{Y} (u)^{1-\theta }}  \label{a}
\end{equation}

Using theorem \ref{Pre_compact}, $\left\{ u(t),t\geq 1\right\} $ compactly lies in $
\mathcal{V}$ so by Theorem 4.3.3 in \cite{Songu}, we can infer that  $\Omega (u_{0})$ is non-empty,
compact, invariant, connected and $dist_{\mathcal{V}}\left( u(t),\Omega (u_{0})\right)
\rightarrow 0$ in norm of $\mathcal{V}.$ Further for all $ x\in \Omega (u_{0})$ we
have $\Upsilon (x)=\,$constant$\,=e_{0}$ and hence
\begin{equation*}
\frac{d\mathcal{Y} (x)}{dt}=0,\text{ for all }x\in \Omega (u_{0}).
\end{equation*}

Moreover, using Theorem \ref{Lwoj}, for all $x\in \Omega (u_{0})$ there is a $\sigma _{0,x}>0$ and $ 0 <\theta _{x} <\frac{1}{2}  $
such that
\begin{eqnarray}
\left\vert  -\Delta^{2} u+2\Delta u+\|u\|^{2}_{\mathcal{V}}u -\|u\|^{2}_{L^{2}}u  +\| u\|^{2n}_{{L}^{2n}} u- u^{2n-1}\right\vert _{\mathcal{H}}
&\geq &\left\vert \mathcal{Y} (u)-\mathcal{Y} (x)\right\vert ^{1-\theta _{x}}\text{ }
\label{ut} \\
\text{for all }u\text{ such that }\left\Vert u-x\right\Vert &<&\sigma _{0,x},
\text{ i.e. for all }u\in B_{\mathcal{V}}\left( x,\sigma _{0,x}\right) .  \notag
\end{eqnarray}

As $\left\{ B_{\mathcal{V}}\left( x,\sigma _{0,x}\right) ,x\in \omega (u_{0})\right\}
$ is an open covering of $\Omega (u_{0})$, and we know that $\Omega (u_{0})$ is
compact so there is $n\in\mathbb{N}$ such that
\begin{equation*}
\left\{ B_{\mathcal{V}}\left( x_{k},\sigma _{0,x_{k}}\right) ,k=1,2,...,n\right\}
\subset \left\{ B_{\mathcal{V}}\left( x,\sigma _{0,x}\right) ,x\in \omega
(u_{0})\right\}
\end{equation*}
is a finite open subcover of $\Omega (u_{0}).$ Set $\Gamma =\cup
_{k=1}^{n}B_{\mathcal{V}}\left( x_k,\sigma _{0,x_k}\right) ,$ indeed $\Omega
(u_{0})\subset \Gamma .$ As $dist_{\mathcal{V}}\left( u(t),\Omega (u_{0})\right)
\rightarrow 0$ so we can determine $t_{0}>0$ such that $\ $for $t>t_{0},$ $
u(t)\in \Gamma .$ Further, we can also determine an increasing sequence $\left\{
t_{k}\right\} _{k\in\mathbb{N}}$, with $t_{1}>t,$ such that for all $t\in \lbrack t_{k},t_{k+1}]$ we have $
u(t)\in B_{\mathcal{V}}\left( x_{i},\sigma _{0,x_{i}}\right) $ for some $i\in\mathbb{N}.$ Set $\theta =\min \left\{ \theta _{x_{1}},\theta _{x_{2}},...,\theta
_{x_{n}}\right\} $. Since $x_{k}\in \Omega (u_{0})$ so $\mathcal{Y} (x_{k})=e_{0}$.

Using inequality \ref{ut}, we get
\begin{equation*}
\left\vert -\Delta^{2} u+2\Delta u+\|u\|^{2}_{\mathcal{V}}u -\|u\|^{2}_{L^{2}}u  +\| u\|^{2n}_{{L}^{2n}} u- u^{2n-1}\right\vert
_{\mathcal{H}}\geq \left\vert \mathcal{Y} (u)-e_{0}\right\vert ^{1-\theta _{x}},\text{ for
all }u\in \Gamma .
\end{equation*}
Using inequality \ref{a}), we have
\begin{eqnarray*}
-\frac{d}{dt}(\mathcal{Y} (u(t))-e_{0})^{\theta _{x}} &=&-\frac{\theta _{x}}{(\mathcal{Y} 
(u(t))-e_{0})^{1-\theta _{x}}}\frac{d}{dt}\left( \mathcal{Y} (u\right) )=\frac{\theta _{x}\left\vert u_t\left( t\right) \right\vert
_{\mathcal{H}}^{2}}{(\mathcal{Y}  (u(t))-e_{0})^{1-\theta _{x}}} \\
&=&\frac{\theta _{x}\left\vert -\Delta^{2} u+2\Delta u+\|u\|^{2}_{\mathcal{V}}u -\|u\|^{2}_{L^{2}}u  +\| u\|^{2n}_{{L}^{2n}} u- u^{2n-1}\right\vert
_{\mathcal{H}}^{2}}{(\mathcal{Y}  (u(t))-e_{0})^{1-\theta _{x}}} \\
&\geq &\frac{\theta _{x}\left( (\mathcal{Y}  (u(t))-e_{0})^{1-\theta _{x}}\right)
^{2}}{(\mathcal{Y}  (u(t))-e_{0})^{1-\theta _{x}}}\geq \theta _{x}\left\vert u_{t}\left( t\right) \right\vert _{\mathcal{H}}^{}.
\end{eqnarray*}

By taking integration on both sides, we get
\begin{eqnarray*}
\theta _{x}\int_{t_{k}}^{t_{k+1}}\left\vert u_{p}\right\vert _{\mathcal{H}}dp &\leq
&-\mathcal{Y} (u(t_{k+1})-e_{0})^{\theta _{x}}+\mathcal{Y} (u(t_{k})-e_{0})^{\theta _{x}}
=-(\mathcal{Y} (u(t_{k+1})-e_{0})^{\theta _{x}}+(\mathcal{Y} (u(t_{k})-e_{0})^{\theta
_{x}}
\end{eqnarray*}
and so
\begin{eqnarray*}
\int_{t_{1}}^{\infty }\left\vert u_{p}\right\vert _{\mathcal{H}}dp
&=&\sum_{k=1}^{\infty }\int_{t_{k}}^{t_{k+1}}\left\vert u_{p}\right\vert
_{\mathcal{H}}dp\leq \theta _{x}^{-1}\sum_{k=1}^{\infty }\left[ (\mathcal{Y}
(u(t_{k})-e_{0})^{\theta _{x}}-(\mathcal{Y} (u(t_{k+1})-e_{0})^{\theta _{x}}\right]
=\theta _{x}^{-1}\left( \mathcal{Y} \left( u\left( t_{1}\right) \right)
-e_{0}\right) ^{\theta _{x}}\\
\left\vert u(s)-u(t)\right\vert _{\mathcal{H}}&\leq& \int_{s}^{t}\left\vert
u_{p}\right\vert _{\mathcal{H}}dp \leq \theta _{x}^{-1}\left( \mathcal{Y} \left( u\left(
s\right) \right) -e_{0}\right) ^{\theta _{x}}.
\end{eqnarray*}
proceeding limit $t\rightarrow \infty $ we infer that
\begin{equation*}
\underset{t\rightarrow \infty }{\lim }u(t)=u^{\infty }\text{ in }\mathcal{L}^{2}.
\end{equation*}
But $dist_{\mathcal{V}}\left( u(t),\Omega (u_{0})\right) \rightarrow 0$ in $\mathcal{V}$, thus this convergence is also true in $\mathcal{V}.$
\end{proof}

Now, we demonstrate another proof of the same Theorem \ref{3.3} using the approach present in \cite{Jandoubi}, introduced by Jendoubi. This method indicates that orbit will utimately reach a point close to stationary solution $\varphi \in \Omega(u_0)$ and remain in that vicinity indefinitely, based on observations.

\begin{proof}

Assume that  $u_{0}\in \mathcal{V}\cap\mathcal{M}$ and $u$ be the global unique solution of our main
problem \ref{PB},with initial data $u_{0}.$ The $\Omega $-limit set mentioned  in Corollary \ref{Coro-PreCompact} can be alternatively written as
\begin{equation*}
\Omega \left( u_{0}\right) =\left\{ \mu \in \mathcal{V}:\text{ there exists }
t_{n}\rightarrow \infty ,\underset{n\rightarrow \infty }{\lim }\left\vert
u(t_{n})-\mu \right\vert _{\mathcal{V}}=0\right\} .\text{ }
\end{equation*}
Suppose $\mu \in \Omega \left( u_{0}\right) $, then there is $t_{n}\rightarrow \infty$ ,
\begin{equation}
\underset{n\rightarrow \infty }{
\lim }\left\vert u(t_{n})-\mu \right\vert _{\mathcal{V}}=0 . \label{a2}
\end{equation}
From our main the problem, we have
\begin{eqnarray*}
\left\vert \frac{du}{dt}\right\vert _{\mathcal{H}}^{2} &=&\left\vert \frac{du}{dt}\right\vert _{\mathcal{H}}\left\vert  -\Delta^{2} u+2\Delta u+\|u\|^{2}_{\mathcal{V}}u -\|u\|^{2}_{L^{2}}u  +\| u\|^{2n}_{{L}^{2n}} u- u^{2n-1}\right\vert _{\mathcal{H}}.
\end{eqnarray*}
And using  Theorem \ref{global_sol_thm},
\begin{equation*}
\left\vert \frac{du}{dt}\right\vert _{\mathcal{H}}^{2}=-\frac{d}{dt}\mathcal{Y} (u).
\end{equation*}

As for each $\varphi \in \Omega \left( u_{0}\right) ,$ $\mathcal{Y}
\left( \varphi \right) $ is constant. From last two equations, we have
\begin{align}
    -\frac{d}{dt}\left[  \mathcal{Y} (u\left( t\right)  )-\mathcal{Y} \left(
\varphi \right) \right] &=-\frac{d}{dt}\left( \mathcal{Y} (u\right) ) \notag \\
&= \left\vert
\frac{du}{dt}\right\vert _{\mathcal{H}}\left\vert -\Delta^{2} u+2\Delta u+\|u\|^{2}_{\mathcal{V}}u -\|u\|^{2}_{L^{2}}u  +\| u\|^{2n}_{{L}^{2n}} u- u^{2n-1}\right\vert _{\mathcal{H}}.  \label{a3}
\end{align}

Now, the function $\left[  \mathcal{Y} (u\left( t_{n}\right)  )-\mathcal{Y} \left(
\mu \right) \right] $ is a scalar function so for any $0<\theta <\frac{1}{2} $ and from equations \eqref{a3} and \eqref
{a2}, we can also derive that , for all $\varepsilon >0,$ there exists $N>0$
such that when $n\geq \mathcal{N}$
\begin{equation}
\left\vert u\left( t_{n}\right) -\mu \right\vert _{\mathcal{V}}<\frac{\varepsilon
}{2}\text{ and }\frac{1}{\theta }\left\vert \left( \mathcal{Y} (u(t_{n}))-\mathcal{Y}
\left( \mu \right) \right) \right\vert ^{\theta }<\frac{\varepsilon }{2},
\text{ for }n\geq \mathcal{N}. \label{a4}
\end{equation}

where $\sigma ,$and $\theta $ are explained in Theorem \ref{Lwoj} and
we select $\mathcal{N}$ sufficiently large so that $\left\vert u\left( t_{n}\right)
-\mu \right\vert _{\mathcal{H}}<\frac{\varepsilon }{2}.$ Define
\begin{equation*}
\tau =\sup \left\{ t\geq t_{\mathcal{N}}:\left\vert u\left( p\right) -\mu
\right\vert _{\mathcal{V}}<\sigma ,\text{ for all }p\in \left[ t_{\mathcal{N}},t\right] \right\}
.
\end{equation*}

Now, we calim that $\tau =\infty .$ Assume on contrary basis that $\tau <\infty$ and
\begin{equation*}
-\frac{d}{dt}\left[ \mathcal{Y} (u(t))-\mathcal{Y} \left( \mu \right) \right] ^{\theta
}=-\theta \frac{d}{dt}\left[ \mathcal{Y} (u(t))-\mathcal{Y} \left( \mu \right) \right]
\frac{d}{dt}\left[ \mathcal{Y} (u(t))-\mathcal{Y} \left( \mu \right) \right] ^{\theta
-1}.
\end{equation*}
Using  equation (\ref{a3}),
\begin{equation*}
-\frac{d}{dt}\left[ \mathcal{Y} (u(t))-\mathcal{Y} \left( \mu \right) \right] ^{\theta
}=-\theta \left\vert \frac{du}{dt}\right\vert _{\mathcal{H}}\left\vert -\Delta^{2} u+2\Delta u+\|u\|^{2}_{\mathcal{V}}u -\|u\|^{2}_{L^{2}}u  +\| u\|^{2n}_{{L}^{2n}} u- u^{2n-1}\right\vert _{\mathcal{H}}\left[
\mathcal{Y} (u(t))-\mathcal{Y} \left( \mu \right) \right] ^{\theta -1}.
\end{equation*}
Using Theorem \ref{Lwoj},
\begin{eqnarray*}
-\frac{d}{dt}\left[ \mathcal{Y} (u(t))-\mathcal{Y} \left( \mu \right) \right] ^{\theta
} &\geq &\theta \left\vert \frac{du}{dt}\right\vert _{\mathcal{H}}\left\vert \mathcal{Y}
(u(t))-\mathcal{Y} \left( \mu \right) \right\vert ^{1-\theta }\left[ \mathcal{Y}
(u(t))-\mathcal{Y} \left( \mu \right) \right] ^{\theta -1} \\
&\geq &\theta \left\vert \frac{du}{dt}\right\vert _{\mathcal{H}}\left[ \mathcal{Y}
(u(t))-\mathcal{Y} \left( \mu \right) \right] ^{1-\theta }\left[ \mathcal{Y}
(u(t))-\mathcal{Y} \left( \mu \right) \right] ^{\theta -1}.
\end{eqnarray*}
Taking Integration on both sides  on $\left( t_{\mathcal{N}},\tau \right) $ w.r.t $t,$
\begin{eqnarray}
\int_{t_{\mathcal{N}}}^{\tau }\left\vert \frac{du}{dt}\right\vert _{\mathcal{H}}dt &\leq &\frac{1
}{\theta }\int_{t_{\mathcal{N}}}^{\tau }\frac{d}{dt}\left[ \mathcal{Y} (u(t))-\mathcal{Y} \left(
\mu \right) \right] ^{\theta }dt  \notag \\
&=&-\frac{1}{\theta }\left( \left[ \mathcal{Y} (u(\tau ))-\mathcal{Y} \left( \mu
\right) \right] ^{\theta }-\left[ \mathcal{Y} (u(t_{\mathcal{N}}))-\mathcal{Y} \left( \mu
\right) \right] ^{\theta }\right)  \notag \\
&\leq &\frac{1}{\theta }\left[ \mathcal{Y} (u(t_{\mathcal{N}}))-\mathcal{Y} \left( \mu \right) 
\right] ^{\theta }.  \label{a5}
\end{eqnarray}
And
\begin{eqnarray*}
\left\vert u\left( \tau \right) -u\left( t_{\mathcal{N}}\right) \right\vert _{\mathcal{H}} &\leq
&\int_{t_{\mathcal{N}}}^{\tau }\left\vert \frac{du}{dt}\right\vert _{\mathcal{H}}dt\leq \frac{1}{
\theta }\left[ \mathcal{Y} (u(t_{\mathcal{N}}))-\mathcal{Y} \left( \mu \right) \right] ^{\theta }
\leq\frac{1}{\theta }\left[ \mathcal{Y} (u(t_{\mathcal{N}}))-\mathcal{Y} \left( \mu \right) \right]
^{\theta } \\
&<&\frac{\varepsilon }{2}+\frac{\varepsilon }{2}=\varepsilon.
\end{eqnarray*}
Using the precompactness assumption, we can infer that  if a sequence converges in $\mathcal{H}$ then it also
converges in $\mathcal{V}$ so
\begin{equation*}
\left\vert u\left( \tau \right) -u\left( t_{\mathcal{N}}\right) \right\vert
_{\mathcal{V}}<\varepsilon <\sigma
\end{equation*}
Thus, it is a  contradiction to our previous claim of $\tau .$ Thus $\tau
=\infty .$ \\
So inequality (\ref{a5}) takes the following form
\begin{eqnarray*}
\int_{t_{\mathcal{N}}}^{\infty }\left\vert \frac{du}{dt}\right\vert _{\mathcal{H}}dt &\leq &
\frac{1}{\theta }\left[ \mathcal{Y} (u(t_{\mathcal{N}}))\right] ^{\theta }\text{ i.e.} \\
\underset{p\rightarrow \infty }{\lim }\int_{t_{N}}^{p}\left\vert \frac{du}{dt
}\right\vert _{\mathcal{H}}dt &\leq &\frac{1}{\theta }\left[ \mathcal{Y} (u(t_{\mathcal{N}}))\right]
^{\theta }
\end{eqnarray*}
\begin{equation*}
\underset{p\rightarrow \infty }{\lim }\left\vert u(p)-u(t)\right\vert
_{\mathcal{H}}\leq \underset{p\rightarrow \infty }{\lim }\int_{t_{\mathcal{N}}}^{p}\left\vert
\frac{du}{dt}\right\vert _{\mathcal{H}}dt\leq \frac{1}{\theta }\left[ \mathcal{Y} (u(t_{\mathcal{N}}))
\right] ^{\theta }<\frac{\varepsilon }{2}.
\end{equation*}

Therfore,  $u(t)$ has a limit in in $\mathcal{H}$. Also using precompactness, $u(t)$ also converges also some point in $\mathcal{V}$
\end{proof}

\section{{\Large Rate of decay to equilibrium}}

This section presents the results about the convergence of the solution to an equilibrium $u^{\infty}$. Furthermore, we will see the convergence of a solution to either an exponential or polynomial decay rate to equilibrium, depending on the values of $\theta$.

\begin{theorem}
Suppose $u^{\infty } \in \Omega(u_0) $ is the limit obtained in Theorem \ref{3.3}. Then following decay rate estimates holds.\\
\textbf{i)} If $\theta=\frac{1}{2}$, then  the solution $u$ of problem \ref{PB} converges to stationary $u^{\infty }$  in $\mathcal{H}$-norm exponentially as $t$ tends to $+\infty$. That is, there are two positive constants $\kappa_1$ and $\kappa_2$  
$$
\left\vert u\left( t \right) -u^{\infty } \right\vert _{\mathcal{H}} \leq \kappa_1 e^{-\kappa_2 t} \qquad \qquad \text { as } t \rightarrow+\infty.
$$
\textbf{ii)} If  $0 <\theta < \frac{1}{2}$, then solution $u$ of problem \ref{PB} converges to stationary $u^{\infty }$ exponentially  in $\mathcal{H}$-norm at polynomial rate as $t$ tends to $+\infty$. That is, there is a positive constant $\kappa$ such that
$$
\left\vert u\left( t \right) -u^{\infty } \right\vert _{\mathcal{H}} \leq \kappa (1+t)^{-\frac{\theta}{1-2 \theta}} \qquad \qquad \text { as } ~~t \rightarrow+\infty.
$$   
\end{theorem}

\begin{proof}
  \textbf{  i)} For the proof of first estimate, we begin by same expression used in theorem (\ref{3.3}). We know that  $\mathcal{Y}$ is constant on $\Omega (u_{0})$ and $u^{\infty } \in \Omega(u_0)$ so $\frac{d}{dt}\mathcal{Y}(u^{\infty })=0$ and there exists a sequence $t_n\rightarrow \infty$ such that $u(t_n)\rightarrow u^{\infty }$ in $\mathcal{H}$ as $n$ tends to  $\infty$.

\begin{eqnarray} \label{4.1}
-\frac{d}{dt}\left(\mathcal{Y} (u(t))-\mathcal{Y} (u^{\infty})\right)^{\theta} &=& -\frac{\theta}{(\mathcal{Y} 
(u(t))-\mathcal{Y} (u^{\infty}))^{1-\theta}}\frac{d}{dt}\left( \mathcal{Y} (u(t))\right) = \frac{\theta\left\vert u_t\left( t\right) \right\vert
_{\mathcal{H}}^{2}}{(\mathcal{Y} (u(t))-\mathcal{Y} (u^{\infty}))^{1-\theta}} \notag \\
&=& \frac{\theta\left\vert -\Delta^{2} u+2\Delta u+\|u\|^{2}_{\mathcal{V}}u -\|u\|^{2}_{L^{2}}u +\| u\|^{2n}_{L^{2n}} u- u^{2n-1}\right\vert
_{\mathcal{H}}^{2}}{(\mathcal{Y} (u(t))-\mathcal{Y} (u^{\infty}))^{1-\theta}} \notag \\
&\geq& \frac{\theta\left( (\mathcal{Y} (u(t))-\mathcal{Y} (u^{\infty}))^{1-\theta}\right)
^{2}}{(\mathcal{Y} (u(t))-\mathcal{Y} (u^{\infty}))^{1-\theta}} \geq K\left\vert u_{t}\left( t\right) \right\vert^{} _{\mathcal{H}}.
\end{eqnarray}
The above inequality \eqref{4.1} takes the form of
\begin{eqnarray}\label{K}
    \frac{d}{dt}\left(\mathcal{Y} (u(t))-\mathcal{Y} (u^{\infty})\right)^{\theta} + K\left\vert u_{t}\left( t\right) \right\vert^{} _{\mathcal{H}} \leq 0
\end{eqnarray}

Taking integration on \eqref{LJ} and  using fact that  $u(t)\rightarrow u^{\infty }$ in $\mathcal{H}$ as $t\rightarrow \infty$, for each $t\geq 0$, we have
\begin{equation}
\left\vert u\left( t \right) -u^{\infty } \right\vert _{\mathcal{H}} \leq \int_t^{+\infty}\left\vert u_{t}^{ }\left( p\right)  \right\vert _{\mathcal{H}} dp  \leq \frac{1}{K} (\mathcal{Y} (u(t))-\mathcal{Y}(u^{\infty }))^{\theta }.\label{df1}
\end{equation}

Using Lojasiewicz–Simon Inequality \eqref{LJ} along with inequality \eqref{K}, we have
$$
\frac{d}{dt}(\mathcal{Y} (u(t))-\mathcal{Y}(u^{\infty }))^{\theta }+K \left((\mathcal{Y} (u(t))-\mathcal{Y}(u^{\infty }))^{\theta }\right)^{\frac{1-\theta}{\theta}} \leq\frac{d}{dt}(\mathcal{Y}(u(t))-\mathcal{Y}(u^{\infty }))^{\theta }+K \left\vert u_t^{ }\left( t\right) \right\vert _{\mathcal{H}}^{}\leq 0 .
$$

That is, 
\begin{equation}
\frac{d}{dt}(\mathcal{Y} (u(t))-\mathcal{Y} (u^{\infty }))^{\theta }+K \left((\mathcal{Y}  (u(t))-\mathcal{Y} (u^{\infty }))^{\theta }\right)^{\frac{1-\theta}{\theta}} \leq 0. \label{df2}   
\end{equation}

If we substitute   $\theta=\frac{1}{2}$, and solve differential inequality \eqref{df2} for $(\mathcal{Y} (u(t))-\mathcal{Y}(u^{\infty }))^{\theta }$, then we can conclude that  there are two positive constants $c_{1}, c_{2}$, such that
\begin{eqnarray}\label{df}
 (\mathcal{Y} (u(t))-\mathcal{Y}(u^{\infty }))^{\theta }\leq   c_1 e^{-c_2 t}. 
\end{eqnarray} 
Using inequalities \eqref{df1} and \eqref{df}, it follows that
\begin{equation}
\left\vert u\left( t \right) -u^{\infty } \right\vert _{\mathcal{H}} \leq  \kappa_1 e^{-\kappa_2 t},   
\end{equation}
where $\kappa_1:=\frac{a}{K}$ and $\kappa_2:=b$, are positive constants.\\

ii) We will use a comparison principle for nonlinear ODEs for the proof decay estimate in the case of $0< \theta<\frac{1}{2}$. Consider the following initial value problem,
\begin{eqnarray}
  &&\frac{d \Lambda}{d t}+K \Lambda^{\frac{1-\theta}{\theta}}=0, \notag \\
&&\Lambda(0)=(\mathcal{Y} (u_0)-\mathcal{Y}(u^{\infty }))^{\theta }>0. \label{df4} 
\end{eqnarray}
The solution of above initial value problem is

$$
\Lambda(t)=\left((\mathcal{Y} (u_0)-\mathcal{Y}(u^{\infty }))^{(2 \theta-1)}+\frac{K(1-2 \theta)}{\theta} t\right)^{-\frac{\theta}{1-2 \theta}}.
$$
Set
$$
\chi(t):=\Lambda(t)-(\mathcal{Y} (u(t))-\mathcal{Y}(u^{\infty }))^{\theta }.
$$
Now, one can easily verify that $\chi$ satisfy the following initial value problem.

$$
\left\{\begin{array}{l}
\frac{d \chi}{d t}+K \left(\frac{1-\theta}{\theta} \eta^{{1-2 \theta}}\right) \chi \geq 0 \\
\chi(0)=0
\end{array}\right.
$$
Where  each $\eta$ satisfying $(\mathcal{Y} (u(t))-\mathcal{Y}(u^{\infty }))^{\theta }<\eta<\Lambda$. Multiplying differential inequality with integrating factor $ \exp{\int_0^t K \left(\frac{1-\theta}{\theta} \eta(p)^{{1-2 \theta}}\right) dp }$ and simplifying, we get  
$$
\chi(t) \exp{\int_0^t K \left(\frac{1-\theta}{\theta} \eta(p)^{{1-2 \theta}}\right) dp } \geq 0 .
$$
Which gives,
$$
\chi(t) \geq 0, \quad \forall t \geq 0 .
$$
That is,
\begin{equation}
\left((\mathcal{Y} (u_0)-\mathcal{Y}(u^{\infty }))^{(2 \theta-1)}+\frac{K(1-2 \theta)}{\theta} t\right)^{-\frac{\theta}{1-2 \theta}} \geq (\mathcal{Y} (u(t))-\mathcal{Y}(u^{\infty }))^{\theta }   
\end{equation}
Using the  inequality \eqref{df1}, we have
\begin{eqnarray}
\left\vert u\left( t \right) -u^{\infty } \right\vert _{\mathcal{H}}\leq\left((\mathcal{Y} (u_0)-\mathcal{Y}(u^{\infty }))^{(2 \theta-1)}+\frac{K(1-2 \theta)}{\theta} t\right)^{-\frac{\theta}{1-2 \theta}}.   \label{df5}
\end{eqnarray}
As from \eqref{df4} we have $(\mathcal{Y} (u_0)-\mathcal{Y}(u^{\infty }))>0$. Moreover, as $\theta\in \left(0,\frac{1}{2}\right)$ so $-1<2\theta-1<0$, which implies that 
\begin{equation}
(\mathcal{Y} (u_0)-\mathcal{Y}(u^{\infty }))^{2\theta-1}<1.    
\end{equation}
Using last estimate, we can rewrite \eqref{df5} as
\begin{equation}
\left\vert u\left( t \right) -u^{\infty } \right\vert _{\mathcal{H}}\leq\left(1+\frac{K(1-2 \theta)}{\theta} t\right)^{-\frac{\theta}{1-2 \theta}}\leq\left(\frac{K(1-2 \theta)}{\theta}\right)^{-\frac{\theta}{1-2 \theta}} \left(\frac{K(1-2 \theta)}{\theta}+ t\right)^{-\frac{\theta}{1-2 \theta}}    \label{df6}    
\end{equation}
Now, take  $K$ such that $\frac{K(1-2 \theta)}{\theta}\leq 1$. Thus, \eqref{df6} becomes
\begin{eqnarray}
\left\vert u\left( \tau \right) -u^{\infty } \right\vert _{\mathcal{H}} \leq \kappa (1+t)^{-\frac{\theta}{1-2 \theta}}    
\end{eqnarray}
where $\kappa:=\left(\frac{K(1-2 \theta)}{\theta}\right)^{-\frac{\theta}{1-2 \theta}} >0$. This is the end of the proof.
\end{proof}

\section{{\Large Existence of Global Attractor}}
In this section, we present the existence of global attractor to our main problem (\ref{PB}).

\begin{theorem}
    The problem (\ref{PB}) admits global attractor $\mathcal{A}=\Omega(u_0)$ in manifold $\mathcal{V}\cap \mathcal{M}$.
\end{theorem}

\begin{proof}

To prove the existence of  global attractor, we will use the definition from Theorem \ref{attractor}. From the  Corollary (\ref{Coro-PreCompact}), we know  that $\{u(t):t\geq 0\}$ is compact in $\mathcal{V}$ so from Corollary \ref{2.2} we can infer that $\mathcal{A}=\Omega(u_0)$ compact, connected invariant set under semigroup. Hence condition (ii) of Theorem \ref{attractor} is fulfilled.\\

From Theorem \ref{global_sol_thm}, we have already the existence of global solution with $\left\Vert u(t)\right\Vert \leq2\mathcal{Y} (u_{0})$. We , now, start by showing the existence of absorbing set in $\mathcal{V}.$  Using the trace Lemma \ref{1-tem}, consider the following chain of equations

\begin{eqnarray*}
\frac{d}{dt}\left\vert u(t)\right\vert _{\mathcal{H}}^{2} &=&  \left\langle u^{\prime }(t),u(t)\right\rangle _{\mathcal{H}}= \left\langle -\Delta^{2} u+2\Delta u+\|u\|^{2}_{\mathcal{V}}u -\|u\|^{2}_{L^{2}}u  +\| u\|^{2n}_{{L}^{2n}} u- u^{2n-1},u(t)\right\rangle _{\mathcal{H}},\\
&=&{\left(\|u\|^{2}_{H^{2}_{0}} + 2 \|u\|^{2}_{H^{1}_{0}} + \|u\|^{2n}_{L^{2n}}\right) \left( |u(t)|^{2}_{\mathcal{H}}-1\right)}\leq\left(\|u\|^{2}_{H^{2}_{0}} + 2 \|u\|^{2}_{H^{1}_{0}} + \|u\|^{2n}_{L^{2n}}\right)  \left\vert u(t)\right\vert _{\mathcal{H}}^{2}. \\
\end{eqnarray*}

Using the continuos embeddings $ \mathcal{V}^{}\hookrightarrow 
 H^{2}_{0}$, $\mathcal{V}\hookrightarrow \mathcal{L}^{2n}$, and $ \mathcal{V}^{}\hookrightarrow 
 H^{1}_{0}$ and the inequality $\left\Vert u(t)\right\Vert_{\mathcal{V}} \leq 2\mathcal{Y} (u_{0})$, we have
\begin{eqnarray*}
\frac{d}{dt}\left\vert u(t)\right\vert _{\mathcal{H}}^{2} &\leq &\left(a^{2}\|u\|^{2}_{\mathcal{V}} + 2 b^2\|u\|^{2}_{\mathcal{V}} + c^{2n}\|u\|^{2n}_{\mathcal{V}}\right)  \left\vert u(t)\right\vert _{\mathcal{H}}^{2} \\
&\leq &\left(4a^{2}(\mathcal{Y}(u_{0}))^2+8b^{2}(\mathcal{Y}(u_{0}))^2+c^{2n}(2c)^{2n}(\mathcal{Y} (u_{0}))^{2n}\right) \left\vert u(t)\right\vert _{\mathcal{H}}^{2}.
\end{eqnarray*}
Using Grownwall's lemma,
\begin{equation*}
\left\vert u(t)\right\vert _{\mathcal{H}}^{2}\leq \left\vert u_{0}\right\vert
_{\mathcal{H}}^{2}\exp \left(4a^{2}(\mathcal{Y}(u_{0}))^2+8b^{2}(\mathcal{Y}(u_{0}))^2+c^{2n}(2c)^{2n}(\mathcal{Y} (u_{0}))^{2n}\right) :=K_{}
\end{equation*}
More precisely,  there is $t_{0}\left( \left\vert u_{0}\right\vert
_{\mathcal{H}}^{2}\right) $ such that
\begin{equation*}
\left\vert u(t)\right\vert _{\mathcal{H}}^{2}\leq K,\text{ for all }t\geq t_{0}.
\end{equation*}

In addition, from Theorem \ref{global_sol_thm}
\begin{equation*}
\frac{d}{dt}\mathcal{Y} (u)=-\left\vert \frac{du}{dt}\right\vert _{\mathcal{H}}^{2}\leq 0
\end{equation*}
So again using Gronwall's inequality, we can deduce that
\begin{equation*}
 \mathcal{Y} (u(t))\leq \mathcal{Y} (u_0). 
\end{equation*}

If the initial data varies within a bounded set in $V$ ( Sobolev imbedding theorem). That is, for a given $\mathcal{R}>0$ 
\begin{equation*}
\mathcal{B}=\{u_0\in\mathcal{V}:\left\Vert u_{0}\right\Vert\leq \mathcal{R}\}.
\end{equation*}
then there is a constant $\kappa(\mathcal{R})>0$ (not on $\Vert u_0 \Vert$) (which depends on the bounded set) such that for all $t\geq t_\kappa$,
\begin{equation*}
 \mathcal{Y} (u(t))\leq \mathcal{Y} (u_0)\leq \kappa. 
\end{equation*}

%%%%%%%%%%%%%%%%%%%%Do from here%%%%%%%%%%%%%%%%%%%%%%%%%%%%

This implies that for $t\geq t_\kappa$ the orbits starting from $\mathcal{B}$ uniformly enter an absorbing ball of radius $\kappa$.

In addition, it is trivial that there exists a  absorbing set in $\mathcal{V}$ norm  because we already know that $\left\Vert
u(t)\right\Vert \leq 2\mathcal{Y} (u_{0}) ,$ so also for $t\geq t\kappa$,
\begin{equation*}
\left\Vert u(t)\right\Vert \leq 2\mathcal{Y} (u_{0})\leq \frac{\kappa}{2}.
\end{equation*}
This implies that orbit  starting from $\mathcal{B}$ enter in ball in $\mathcal{V}$ with radius $\frac{\kappa}{2}$. Thus, condition (i) of Theorem \ref{attractor} is fulfilled. and global attractor exists.
\end{proof}

\section{{\Large Conclusion}}

This paper presents the dynamical properties and long-term behavior  of solution of particular gradient flow with the solution in Hilbert manifold. Firstly, the convergence of global solution of the main problem (\ref{PB}) to equilibrium by means of Lojasiewicz–Simon Inequality has been proven. Secondly, the decay rate of convergence of solution to equilibrium have been analyzed. In this case, we have shown that the decay rate is either polynomial or exponential, depending on the values of $\theta$. Finally, the existence of a global attractor has been established. \\ \\
\textbf{Declarations}\\ \\
\textbf{Ethical approval}\\ 
Not applicable.\vspace{0.5cm}\\
\textbf{Competing interests}\\ 
The author declare having no competing interests.
\vspace{0.5cm}\\ 
\textbf{Authors' contributions}\\ 
All work has been done by solo author.
\vspace{0.5cm}\\ 
\textbf{Funding}\\ 
Not applicable. \vspace{0.5cm}\\
\textbf{Availability of data and materials}\\ 
Not applicable.

\end{document}